\documentclass[reqno]{amsart}
\usepackage{amsmath, amsthm, amssymb, amstext}

\usepackage[left=3cm,right=3cm,top=3cm,bottom=3cm]{geometry}
\usepackage{hyperref,xcolor}
\hypersetup{pdfborder={0 0 0},colorlinks}
\usepackage{enumitem}
\allowdisplaybreaks
\usepackage{todonotes}

\newtheorem{theorem}{Theorem}
\newtheorem{remark}[theorem]{Remark}
\newtheorem{lemma}[theorem]{Lemma}
\newtheorem{proposition}[theorem]{Proposition}

\newtheorem{definition}[theorem]{Definition}
\newtheorem{example}[theorem]{Example}

\DeclareMathOperator*{\esssup}{ess\,sup}
\DeclareMathOperator*{\essinf}{ess\,inf}

\renewcommand{\L}{\left}
\renewcommand{\r}{\right}
\newcommand{\abs}[1]{\left\lvert#1\right\rvert}
\newcommand{\norm}[1]{|\!|#1|\!|}

\newcommand{\curly}[1]{\left\{#1\right\}}

\newcommand{\round}[1]{\left(#1\right)}

\newcommand{\scal}[1]{\left\langle#1\right\rangle}

\renewcommand{\l }{\lambda }
\renewcommand{\O }{\Omega }
\newcommand{\h}{\mathcal{H}}
\newcommand{\R}{{\mathbb R}}
\newcommand{\N}{{\mathbb N}}
\newcommand{\RN}{\mathbb{R}^d}

\newcommand{\Wr}{W^{1, \mathcal{H}}_{\textrm{\operatorname{rad}}}(\RN)}

\def\ds{\displaystyle}
\def\w{W^{1,\mathcal{H}}_0(\Omega)} \def\W{W^{1,\mathcal{H}}(\RN)}
\def\ww{W^{1,\mathcal{H}}(\Omega)}
\def\lh{L^{\mathcal{H}}(\RN)}
\def\WV{W^{1,\mathcal{H}}_V(\R^d)}

\def\Wr{W^{1,\mathcal{H}}_{\operatorname{rad}}(\R^d)}
\def\vi{\rho_{\mathcal{H}}}
\def\Lh{L^{\mathcal{H}}(\RN)}

\numberwithin{theorem}{section}
\numberwithin{equation}{section}

\title[Double phase problems with variable exponents]{Double phase problems with variable exponents depending on the solution and the gradient in the whole space $\mathbb{R}^N$}

\author[A.E. Bahrouni]{Ala Eddine Bahrouni}
\address[A.E. Bahrouni]{Mathematics Department, Faculty of Sciences, University of Monastir, 5019 Monastir, Tunisia}
\email{ala.bahrouni@fsm.rnu.tn}

\author[A. Bahrouni]{Anouar Bahrouni }
\address[A. Bahrouni]{Mathematics Department, Faculty of Sciences, University of Monastir, 5019 Monastir, Tunisia}
\email{bahrounianouar@yahoo.fr}

\author[P. Winkert]{Patrick Winkert}
\address[P. Winkert]{Technische Universit\"{a}t Berlin, Institut f\"{u}r Mathematik, Stra\ss e des 17.\,Juni 136, 10623 Berlin, Germany}
\email{winkert@math.tu-berlin.de}


\begin{document}

\begin{abstract}
	In this paper, we establish continuous and compact embeddings for a new class of Musielak-Orlicz Sobolev spaces in unbounded domains driven by a double phase operator with variable exponents that depend on the unknown solution and its gradient. Using these embeddings and an abstract critical point theorem, we prove the existence and multiplicity of weak solutions for such problems associated with this new operator in the whole space $\mathbb{R}^d$. This work can be seen as a continuation of the recent paper by Bahrouni--Bahrouni--Missaoui--R\u{a}dulescu \cite{Bahrouni-Bahrouni-Missaoui-Radulescu-2024}.
\end{abstract}

\maketitle

\section{Introduction}

Given $d \geq 3$, this paper is concerned with the following problem
\begin{equation}\label{prb}
	\begin{aligned}
		-\operatorname{div}&\left(|\nabla u|^{p(x,|\nabla u|)-2} \nabla u+ \mu(x)| \nabla u|^{q(x,|\nabla u|)-2}  \nabla u\right)\\&+ V(x)\left(|u|^{p(x,| u|)-2}  u+ \mu(x)|  u|^{q(x,| u|)-2}  u\right)=\l f(x,u), \quad x \in \mathbb{R}^d,
	\end{aligned}
\end{equation}
where $p, q, f\colon \mathbb{R}^d \times \mathbb{R} \rightarrow \mathbb{R}$ are three Carathéodory functions satisfying certain assumptions, $V\colon \mathbb{R}^d \rightarrow \mathbb{R}$ is a measurable function, $0 \leq \mu(\cdot) \in C^{0,1}(\mathbb{R}^N)$, and $\lambda$ is a positive constant.

Problems with functionals whose growth conditions depend on the solution or its gradient have been studied very effectively in various applications, including digital image denoising, non-Newtonian fluid flow through porous media, phase transitions, and fluid dynamics, among others, see the works by Blomgren--Chan--Mulet--Wong \cite{Blomgren-Chan-Mulet-Wong-1997}, Blomgren--Chan--Mulet--Vese--Wan \cite{Blomgren-Chan-Mulet-Vese-Wan-2000}, Bollt--Chartrand--Esedo\={g}lu--Schultz--Vixie \cite{Bollt-Chartrand-Esedoglu-Schultz-Vixie-2009} and Bahrouni--Bahrouni--Missaoui--R\u{a}dulescu \cite{Bahrouni-Bahrouni-Missaoui-Radulescu-2024} for an overview. These problems generally focus on minimizing functionals where the growth behavior is closely related to the size of the gradient or the solution itself. This paper can be seen as a continuation of the work by Bahrouni--Bahrouni--Hlel Missaoui--R\u{a}dulescu \cite{Bahrouni-Bahrouni-Missaoui-Radulescu-2024}, extending the analysis made in their work to the entire space $\RN$. In particular,  we establish fundamental results concerning the continuity and compactness of embeddings for the associated function spaces related to our operator. Specifically, we prove the important continuous and compact embeddings, as outlined in Theorems \ref{Inj1}, \ref{Inj2}, \ref{Inj3}, \ref{Inj30} and \ref{thms}.

The study of nonlinear problems in unbounded domains in $\mathbb{R}^d$, such as our problem \eqref{prb}, presents significant challenges, primarily due to the lack of compactness in the embedding of the solution space into appropriate Lebesgue-type spaces. This issue creates substantial difficulties in establishing existence and multiplicity results for possible solutions. Moreover, the specific dependence of the exponents on the unknown solution and its gradient further complicates the analysis. These challenges require the development of a refined and innovative technical framework to support our main findings.

Problem \eqref{prb} is driven by the following new double phase type operator
\begin{align}\label{oper}
	u\longmapsto -\operatorname{div}\L(a(x,|\nabla u|)\nabla u \r) + V(x)a(x,|u|)u,
\end{align}
where  $a\colon  \mathbb{R}^d \times \mathbb{R} \to \mathbb{R}$ is defined as
\begin{align}\label{A}
	a(x,\xi) =|\xi|^{p(x, |\xi|)-2} + \mu(x) |\xi|^{q(x, |\xi|)-2}.
\end{align}
The main feature of the operator given in \eqref{oper} is its ability to switch seamlessly between the regions $\{x \in \mathbb{R}^d \colon \mu(x) = 0\}$ and $\{x \in \mathbb{R}^d \colon \mu(x) > 0\}$. This continuous switching behavior is what gives it the name ``double phase'' operator.

In the last decade, the study of double phase problems has attracted great interest due to their ability to model complex phenomena and materials with heterogeneous properties. Double phase operators, such as
\begin{align}\label{eqc}
	-\operatorname{div} \left( |\nabla u|^{p-2} \nabla u + \mu(x) |\nabla u|^{q-2} \nabla u \right)
\end{align}
(considering the case when $p(x,t)=p$ and $q(x,t)=q$ in our operator), with the related energy functional
\begin{align}\label{dp.f}
	I(u)=\int_\Omega \left(\frac{|\nabla u|^p}{p} + \mu(x)\frac{|\nabla u|^q}{q}\right)  \,\mathrm{d} x,
\end{align}
with $\Omega$ being a domain in $\mathbb{R}^d$, were first introduced by Zhikov \cite{Zhikov-1995,Zhikov-1997} (see also Zhikov--Kozlov--Ole\u{\i}nik \cite{Zhikov-Kozlov-Oleinik-1994}) in connection with the investigation of materials with strong anisotropy. The differential operator and the energy functional given in \eqref{eqc} and \eqref{dp.f} appear in various physical applications, such as transonic flows (see Bahrouni--R\u{a}dulescu--Repov\v{s} \cite{Bahrouni-Radulescu-Repovs-2019}, quantum physics (see Benci--D'Avenia--Fortunato--Pisani \cite{Benci-DAvenia-Fortunato-Pisani-2000}), reaction-diffusion systems (see Cherfils--Il'yasov \cite{Cherfils-Ilyasov-2005}), and non-Newtonian fluids (see Liu--Dai \cite{Liu-Dai-2018-2}).  There is now a considerable amount of literature and a growing interest in double phase equations governed by the operator  \eqref{eqc}. While it is impossible to provide a comprehensive list, some pioneering works on the existence, nonexistence, and uniqueness of solutions include the works by Ambrosio--Essebei \cite{Ambrosio-Essebei-2023}, Arora--Fiscella--Mukherjee--Winkert \cite{Arora-Fiscella-Mukherjee-Winkert-2023}, Biagi--Esposito--Vecchi \cite{Biagi-Esposito-Vecchi-2021}, Colasuonno--Squassina \cite{Colasuonno-Squassina-2016}, Gasi\'nski--Papageorgiou \cite{Gasinski-Papageorgiou-2019}, Gasi\'nski--Winkert \cite{Gasinski-Winkert-2020}, Ge--Pucci \cite{Ge-Pucci-2022}, Liu--Dai \cite{Liu-Dai-2018-1,Liu-Dai-2020, Liu-Dai-2018-2},  Liu--Winkert \cite{Liu-Winkert-2022}, Mingione--R\u{a}dulescu \cite{Mingione-Radulescu-2021}, Papageorgiou--Pudelko--R\u{a}dulescu \cite{Papageorgiou-Pudelko-Radulescu-2023}, see also the references therein. In addition, further studies are carried out that deal with regularity results for minimizers of the functional \eqref{dp.f}. Such results can be found in the papers by Baroni--Colombo--Mingione \cite{Baroni-Colombo-Mingione-2015}, Colombo--Mingione \cite{Colombo-Mingione-2015}, De Filippis--Mingione \cite{DeFilippis-Mingione-2021-1, DeFilippis-Mingione-2021-2, DeFilippis-Mingione-2020-1, DeFilippis-Mingione-2020-2}, De Filippis--Oh \cite{DeFilippis-Oh-2019} and De Filippis--Palatucci \cite{DeFilippis-Palatucci-2019}.

The study of the double phase operator and the associated function space was recently advanced by Crespo–Blanco--Gasi\'nski--Harjulehto--Winkert \cite{Crespo-Blanco-Gasinski-Harjulehto-Winkert-2022}. Specifically, the authors investigate a quasilinear elliptic equation formulated by the following double phase operator with variable exponents:
\begin{align}\label{I4}
	A(u)=-\operatorname{div}\left(|\nabla u|^{p(x)-2} \nabla u+ \mu(x)|\nabla u|^{q(x)-2} \nabla u\right), \quad\text{for all } u \in \ww,
\end{align}
with $p,\ q \in C(\overline{\O})$ such that $1 < p(x) < d,\  p(x) < q(x)$ for all $x \in \O$ and $0 \leq \mu(\cdot) \in L^1(\O)$ and $\ww$ is the corresponding Musielak–Orlicz Sobolev space being a uniformly convex space. The paper by Crespo–Blanco--Gasi\'nski--Harjulehto--Winkert \cite{Crespo-Blanco-Gasinski-Harjulehto-Winkert-2022} was fundamental in this field and led to further studies on the existence, multiplicity and regularity of solutions to different types of problems driven by the operator $A$ using various methods and techniques. We emphasize the works by Arora--Dwivedi \cite{Arora-Dwivedi-2023} and Liu--Dai--Papageorgiou--Winkert \cite{Liu-Dai-Papageorgiou-Winkert-2022}, who used the fibering method and the Nehari manifold to establish the existence of at least two weak solutions for singular double phase problems. Ha--Ho \cite{Ha-Ho-2025} then introduced a Lions-type concentration-compactness principle for spaces associated with the double phase operator defined in \eqref{I4} to overcome the lack of compactness which is caused by the presence of a critical exponent in the nonlinear data and allows the existence and concentration of solutions to be investigated. Subsequently, Ho--Winkert \cite{Ho-Winkert-2023} applied an abstract critical point result due to Kajikiya \cite{Kajikiya-2005} along with recent a priori bounds to demonstrate a sequence of nontrivial solutions to Kirchhoff double phase problems with variable exponents. Kim--Kim--Oh--Zeng \cite{Kim-Kim-Oh-Zeng-2022} proved the existence and multiplicity of solutions to concave-convex-type double phase problems, showing that these solutions are bounded. Finally, Vetro--Winkert \cite{Vetro-Winkert-2023} used truncation arguments and comparison methods to prove the existence of at least two constant-sign solutions.

Recently, Bahrouni--Bahrouni--Missaoui--R\u{a}dulescu \cite{Bahrouni-Bahrouni-Missaoui-Radulescu-2024} investigated problems involving the operator $\operatorname{div} \L(a(x, |\nabla u|)\nabla u\r)$ in a bounded Lipschitz domain $\Omega$ with $a$ as given in \eqref{A}. The novelty of this work, in comparison to the double phase operator with variable exponents \eqref{I4}, lies in the fact that the exponents depend on the gradient of the solution. To the best of our knowledge, it is noteworthy that the paper in \cite{Bahrouni-Bahrouni-Missaoui-Radulescu-2024} was the first to investigate this problem in a bounded domain with variable exponents whereby this is a special case. Under certain assumptions on $p$, $q$, and $\mu$, the authors explored the properties of the Musielak-Orlicz spaces $L^{\h}(\O)$ and the Musielak-Orlicz Sobolev spaces $W^{1,\h}(\O)$, where $\h \colon \Omega \times [0,+\infty) \to \mathbb{R}$ is defined by
\begin{equation}\label{hdfn}
	\h (x,t) = \int_{0}^{t} h(x,s)\,\mathrm{d} s,
\end{equation}
and  $h\colon \mathbb{R}^d \times \mathbb{R} \to \mathbb{R}$ is defined by
\begin{align}\label{h}
	h(x,t) :=
	\begin{cases}
		a(x,|t|)t, & \text{for } t \neq 0, \\
		0, & \text{for } t=0.
	\end{cases}
\end{align}
with $a(\cdot,\cdot)$ as defined in \eqref{A}.
This generalized \textnormal{N}-function (see  Definition \ref{dfnfct}) is of particular importance due to the novel dependence of the variable exponents on the unknown solution and its gradient. This dependence creates new complexities and challenges in the analysis of such operators. Using the recent embedding results in Musielak-Orlicz Sobolev spaces proven by Cianchi--Diening \cite{Cianchi-Diening-2024}, Bahrouni--Bahrouni--Missaoui--R\u{a}dulescu \cite{Bahrouni-Bahrouni-Missaoui-Radulescu-2024} obtained several continuous and compact embedding results for the spaces $\ww$ and $\w$ into Musielak spaces, where $\O$ is bounded Lipschitz domain in $\RN$.

Building on the results of the paper \cite{Bahrouni-Bahrouni-Missaoui-Radulescu-2024}, we perform a deeper analysis of the operator \eqref{oper} in the whole space, using the \textnormal{N}-function $\h$ with variable exponents as a basis. The exponents $p$ and $q$ satisfy the following hypotheses:
\begin{enumerate}[label=\textnormal{(H)},ref=\textnormal{H}]
	\item\label{H}
		\begin{enumerate}
			\item[\textnormal{(i)}]
				Functions $p$ and $q$ are bounded, that is,
				\begin{align*}
					2&\leq p^-:= \essinf_{(x,t)\in \RN\times \R_+}p(x,t) \leq p^+:=\esssup_{(x,t)\in \RN\times \R_+}p(x,t)< d,\\
					2&\leq q^-:= \essinf_{(x,t)\in \RN\times \R_+} q(x,t)\leq q^+:=\esssup_{(x,t)\in \RN\times \R_+}q(x,t)<+\infty,
				\end{align*}
				and
				\begin{align*}
					p(x,t)< q(x,t)< p^-_\ast:=\frac{dp^-}{d-p^-} \quad\text{for a.a.\,}x\in \RN \text{ and for all }t\geq0.
				\end{align*}
			\item[\textnormal{(ii)}]
				The functions $p$ and $q$ are two Carath\'{e}odory functions. Furthermore, they exhibit a constant behavior equal $p(x)=p(x,1)$ and $q(x)=q(x,1)$, respectively, for all $t \in [0,1]$ and a nondecreasing behavior for $t \geq 1$.
			\item[\textnormal{(iii)}]
				There exist constant $c_p, c_q>0$ such that
				\begin{align*}
					|p(x,t)-p(y,t)| \leq c_p|x-y|
					\quad\text{and}\quad
					|q(x,t)-q(y,t)| \leq c_q|x-y|,
				\end{align*}
				for all $t \geq0$ and for a.a.\,$x,y \in \RN$.
			\item[\textnormal{(iv)}]
				$\dfrac{q^+}{p^-}<1+\dfrac{1}{d}$.
	\end{enumerate}
\end{enumerate}

The first objective of this paper is to extend the work of Bahrouni--Bahrouni--Missaoui--R\u{a}dulescu \cite{Bahrouni-Bahrouni-Missaoui-Radulescu-2024} by investigating the embedding properties of the associated function spaces in $\mathbb{R}^d$. First, we want to establish new continuous embedding results for Musielak-Orlicz Sobolev spaces $\W$ into $L^{\h_*}(\mathbb{R}^d)$, where $\h_\ast$ is the Sobolev conjugate of $\h$ (see Definition \ref{csh}). We must first give the following definition.

\begin{definition}
	We say that a generalized \textnormal{N}-function $\h$ satisfies the boundedness condition if  there exist $C_1,C_2>0$ such that
	\begin{equation}\label{bf}
		C_1\leq  \h(x,1)\leq C_2\quad\text{for all } x\in \Omega.\tag{$\mathcal{B}$}
	\end{equation}
\end{definition}

Our first result for continuous embeddings reads as follows.

\begin{theorem}[Continuous embedding]\label{Inj1}
	Let hypotheses \eqref{H} be satisfied. Then, the following hold:
	\begin{enumerate}
		\item[\textnormal{(i)}]
			the embedding $\W \hookrightarrow  L^{\h_\ast}(\RN)$ is continuous;
		\item[\textnormal{(ii)}]
			for any generalized \textnormal{N}-function $\mathcal{V}$  satisfying \eqref{bf},
			\begin{align}\label{cA}
				1< v^- \leq \frac{v(x,t)t}{\mathcal{V}(x,t)} \leq v^+ <+\infty\quad \text{for all } x \in \RN \text{ and for all } t\geq0,
			\end{align}
			with $ \mathcal{V}(x,t)=\int_{0}^{t} v(x,s)\,\mathrm{d}s,$
			\begin{equation}\label{2eq60}
				\mathcal{V} \ll \h_*,
			\end{equation}
			where $\ll$ is defined in Definition \ref{ddffd} and
			\begin{equation}\label{mla1b}
				\lim_{|t|\rightarrow 0} \frac{\mathcal{V}(x,t)}{\h (x,t)}=0 \quad\text{uniformly in } x \in  \RN,
			\end{equation}
	\end{enumerate}
	the embedding
	\begin{align*}
		\W \hookrightarrow  L^{\mathcal{V}}(\RN)
	\end{align*}
	is continuous.
\end{theorem}

The above theorem is new to the literature, and we have drawn inspiration from the recent work of Cianchi--Diening \cite{Cianchi-Diening-2024} to get our continuous embedding results for Musielak-Orlicz Sobolev spaces in $\mathbb{R}^d$. It is important to note that in our case the explicit form of $\h$ is unknown, which poses an additional challenge when setting up theorem \ref{Inj1}.

The main difficulty in the study of problem \eqref{prb} on the whole space $\RN$ is the lack of compactness. To overcome this, one can use the method of unbounded potential, i.e.\,the problem contains a term with unbounded potential. Next, we introduce the hypotheses on the potential function $V$:

\begin{enumerate}[label=\textnormal{$(V_0)$},ref=\textnormal{$V_0$}]
	\item\label{V0}
		\begin{enumerate}
			\item[\textnormal{(i)}]
				$V\in C(\RN,\R)$  and  there exists $V_0 > 0$ such that $V (x) \geq V_0$ for any $x \in  \RN$;
			\item[\textnormal{(ii)}]
				the set $\curly{x \in \RN \colon V (x) < L}$ has finite Lebesgue measure for each $L > 0$.
		\end{enumerate}
\end{enumerate}

\begin{enumerate}[label=\textnormal{$(V_1)$},ref=\textnormal{$V_1$}]
	\item\label{V1}
		\begin{enumerate}
			\item[\textnormal{(i)}]
				the same as $\eqref{V0}(\mathrm{i})$;
			\item[\textnormal{(ii)}]
				$\ds\lim_{|x| \rightarrow+ \infty} \ds \int_{B_1(x)} \frac{1}{V(y)} \,\mathrm{d}y =0$ with $B_1(x)=\curly{y\in\RN\colon  |y-x|<1}$.
		\end{enumerate}
\end{enumerate}

Next, we define the following subspace of $\W$:
\begin{align*}
	\WV:= \curly{ u \in \W\colon \int_{\RN}^{} V(x) \h(x,u)\,\mathrm{d}x <\infty},
\end{align*}
endowed with the norm
\begin{align*}
	\norm{u}_{\WV} := \norm{\nabla u}_{\lh}+ \norm{u}_{L^{\h}_{V}(\RN)},
\end{align*}
with $\norm{u}_{\W}$ as in \eqref{nhr} and
\begin{align*}
	\norm{u}_{L^{\h}_{V}(\RN)}:= \inf \curly{ \l >0\colon \ \int_{\RN}^{} V(x) \h\round{x, \frac{u(x)}{\l}} \leq 1 }.
\end{align*}

Our main results about compact embeddings are given in the next theorems.

\begin{theorem}[Compact embedding]\label{Inj2}
	Let hypotheses \eqref{H}  and \eqref{V0} be satisfied. Then, the embedding $ \WV \hookrightarrow L^{\h}(\RN) $ is compact.
\end{theorem}

As a consequence of the last theorem, we can give the following result.

\begin{theorem}[Compact embedding]\label{Inj3}
	Assume that conditions \eqref{H} and \eqref{V0} hold. Let $\mathcal{V}$ be a generalized \textnormal{N}-function that satisfies \eqref{bf}, \eqref{cA}, \eqref{2eq60}, and at least one of the following conditions:
	\begin{enumerate}
		\item[\textnormal{(1)}]
			It holds
			\begin{align}\label{mla1}
				\limsup_{\vert t \vert \rightarrow 0}\frac{\mathcal{V}(x,\vert t \vert)}{\h(x,\vert t \vert)} < +\infty \quad\text{uniformly in } x \in \mathbb{R}^d. \tag{$B_1$}.
			\end{align}
		\item[\textnormal{(2)}]
			There exists $a \in (0,1)$ such that
			\begin{align}\label{mla2}
				\mathcal{V}(x,\vert t \vert) \leq  \h(x,|t|)^a \h_*(x,\vert t \vert)^{1-a}\quad \text{for all }  \vert t \vert \leq 1  \text{ and all }  x \in \mathbb{R}^d.
			\end{align}
	\end{enumerate}
	Then, the space $\WV$ is compactly embedded into $L^{\mathcal{V}}(\mathbb{R}^{d})$.
\end{theorem}

Consequently, we can establish the following result.

\begin{theorem}[Compact embedding]\label{Inj30}
	Assume that conditions \eqref{H} and \eqref{V1} hold.
	\begin{enumerate}
		\item[\textnormal{(i)}]
			The continuous embedding $\WV \hookrightarrow L^r(\RN)$ holds for $r \in [p^-, p^-_\ast]$, where $p^-_\ast = \frac{dp^-}{d - p^-}$.
		\item[\textnormal{(ii)}]
			The compact embedding $\WV \hookrightarrow L^r(\RN)$ holds for $r \in [p^-, p^-_\ast[$.
	\end{enumerate}
\end{theorem}

Let us discuss the significance of the assumptions \eqref{V0} and  \eqref{V1}. The hypotheses \eqref{V0} was introduced by Bartsch--Wang \cite{Bartsch-Wang-1995}, which ensures the compact embedding of the solution space into Lebesgue space while the assumptions in \eqref{V1} has been introduced by Benci--Fortunato \cite{Benci-Fortunato-1978} in which the authors studied the spectrum of the Schr\"{o}dinger operator with potential $V$. In 2002, Salvatore \cite[ Proposition 3.1 and Remark 3.2]{Salvatore-2003} make a comparison between the two hypotheses and conclude that the assumption \eqref{V1} is strictly weaker than \eqref{V0}. The work by Bartsch--Wang \cite{Bartsch-Wang-1995} has inspired numerous studies on nonlinear equations involving various operators and unbounded potentials. In this context, we cite some relevant works and their references that were fundamental to our study. Under assumption \eqref{V0}, Silva--Carvalho--de Albuquerque--Bahrouni \cite{Silva-Carvalho-deAlbuquerque-Bahrouni-2021} demonstrated continuous and compact embedding results for fractional Orlicz spaces. Similar results are presented by Bahrouni--Missaoui--Ounaies \cite{Bahrouni-Missaoui-Ounaies-2024} for fractional Musielak-Orlicz Sobolev spaces. Regarding assumption \eqref{V1}, Bartolo--Candela--Salvatore \cite{Bartolo-Candela-Salvatore-2016} examined continuous and compact embeddings of the weighted Sobolev space $W^{1,p}_V(\mathbb{R}^d)$ into Lebesgue spaces. Building on this, Stegli\'{n}ski \cite{Steglinski-2022}, established an embedding result for the weighted Musielak-Orlicz Sobolev spaces associated with the classical double phase \textnormal{N}-function (related to the operator \eqref{eqc}). Theorems \ref{Inj2}, \ref{Inj3} and \ref{Inj30} are a further development of the previous work by dealing with the new double phase problem \eqref{prb}, where the exponents depend on both the solution and its gradient.

In the next result, we denote by $W^{1,\h}_{\operatorname{rad}}(\mathbb{R}^d)$ the subspace of $\W$ consisting of all radial functions and we present a key result concerning the embedding theory within Musielak-Orlicz Sobolev spaces, with a particular focus on radially symmetric functions. This theorem offers significant insights into the behavior of these functions and their embeddings into other Musielak-Orlicz spaces.

\begin{theorem}[Strauss radial embedding]\label{thms}
	Let condition \eqref{H} be satisfied and $\mathcal{V}$ be a generalized \textnormal{N}-function satisfying \eqref{bf}, \eqref{cA}, \eqref{2eq60}, and \eqref{mla1b}. Then, the space $W^{1,\h}_{\operatorname{rad}}(\mathbb{R}^d)$ is compactly embedded into $L^{\mathcal{V}}(\mathbb{R}^d)$.
\end{theorem}

It is generally recognized in the literature that Strauss' radial embedding results provide an alternative approach to solving problems with lack of compactness, especially when the potential $V$ is constant. The pioneering work in this field is attributed to Strauss \cite{Strauss-1977}, where the potential $V(\cdot)$ is constant. In this groundbreaking work, Strauss introduced the well-known compactness lemma for radially symmetric functions, which has since become a fundamental tool in analyzing such problems. The unknown specific form of the function $\h$ and the dependence of the exponents $p$ and $q$ both on the unknown solution and its gradient make it difficult to apply classical techniques to prove the Strauss embedding theorem. As a result, it is necessary to extend the well-known ``Lions Lemma'' to the new Musielak-Orlicz Sobolev spaces in order to establish the theorem \ref{thms}. Note that Liu--Dai \cite{Liu-Dai-2020} have established a version of the theorem \ref{thms} with a constant variable exponent by comparing $W^{1,\h}_{\operatorname{rad}}(\mathbb{R}^d)$ with $W^{1,p}_{\operatorname{rad}}(\mathbb{R}^d)$.

However, in Theorem \ref{thms} we prove a generalized result for compact embeddings for our radial function space, not in a Lebesgue space, but in a generalized Musielak-Orlicz space under certain assumptions, despite the challenges posed by the dependence of the exponents on the solution and its gradient.

\begin{theorem}[Lions-type lemma]\label{lions}
	Let \eqref{H} be satisfied and $\mathcal{V}$ be a generalized \textnormal{N}-function satisfying \eqref{bf}, \eqref{cA}, \eqref{2eq60}, and \eqref{mla1b}. Furthermore, let $\lbrace u_{n} \rbrace_{n \in \mathbb{N}}$ be a bounded sequence in $W^{1,\h}(\mathbb{R}^{d})$ such that $u_n \rightharpoonup 0$ in $\W$ and
	\begin{align}\label{lionss}
		\lim_{n \rightarrow +\infty} \left[ \sup_{y \in \mathbb{R}^{d}} \int_{B_{r}(y)} \h(x,u_{n})\,\mathrm{d}x \right] = 0\quad\text{for some } r > 0.
	\end{align}
	Then, $u_{n} \rightarrow 0$ in $L^{\mathcal{V}}(\mathbb{R}^{d})$.
\end{theorem}

Having obtained the embedding results, we apply them to study the existence and multiplicity of weak solutions of equation \eqref{prb} involving the double phase operator \eqref{oper}. This is our last main goal in this paper.  Our main tool in this investigation is a theorem on the existence of two critical points established by Bonanno--D'Agu\`\i{} \cite[Theorem 2.1 and Remark 2.2]{Bonanno-DAgui-2016}. This theorem is an important consequence of Bonanno's local minimal theorem \cite[Theorem 2.3]{Bonanno-2012} in conjunction with the Ambrosetti--Rabinowitz theorem.

For the the nonlinearity $f$, we suppose the following hypotheses.

\begin{enumerate}[label=\textnormal{(F)},ref=\textnormal{F}]
	\item\label{F}
		$f\colon \RN \times \R \rightarrow \R$ is a Carath\'{e}odory function (i.e.\,$f(\cdot, t)$ is measurable in $\RN$ for all $t\in  \R$ and $f(x, \cdot)$ is continuous in $\R$ for a.a.\,$x \in \RN$ ), such that
		\begin{enumerate}
			\item[\textnormal{(i)}]
				There  exist  a generalized \textnormal{N}-function $\mathcal{B}(x,t)=\int_{0}^t b(x,s) \,\mathrm{d}s$ such that
				\begin{align*}
					\vert f(x,t)\vert \leq b(x,\vert t\vert) \quad\text{for all }t\in \mathbb{R} \text{ and for a.a.\,}x\in \RN,
				\end{align*}
				and
				\begin{align*}
					q^+\leq b^-:=\inf\limits_{t>0}\frac{b(x,t)t}{\mathcal{B}(x,t)}\leq b^{+}:=\sup\limits_{t>0}\frac{b(x,t)t}{\mathcal{B}(x,t)}<p^-_\ast \quad \text{for a.a.\,} x\in \RN.
				\end{align*}
			\item[\textnormal{(ii)}]
				$\displaystyle{\lim\limits_{t\rightarrow \pm\infty}\frac{F(x,t)}{ \vert t\vert^{q^+}}=+\infty}$ uniformly in $x \in \R$, where $\displaystyle{F(x,t)=\int^t_0 f(x,s)\,\mathrm{d}s}$.
			\item[\textnormal{(iii)}]
				$f(x, t)=o\left(|t|^{p^{-}-1}\right)$ as $|t| \rightarrow 0$ uniformly in $x \in \RN$.
			\item[\textnormal{(iv)}]
				$\tilde{F}(x, t)=\frac{1}{q^{+}} f(x, t) t-F(x, t)>0$ for $|t|$ large and there exist constants $\sigma >\frac{d}{p^{-}}$, $\tilde{c}>0$ and $r_0>0$, such that
				\begin{align}\label{ex0}
					|f(x, t)|^\sigma \leq \tilde{c}|t|^{\left(p^{-}-1\right) \sigma} \tilde{F}(x, t) \quad \text{for all }(x, t) \in \RN \times \mathbb{R} \text{ with }|t| \geq r_0.
				\end{align}
		\end{enumerate}
\end{enumerate}

\begin{example}
	Let $f\colon \mathbb{R}^d \times \mathbb{R} \to \mathbb{R}$ be the function defined by
	\begin{align*}
		f(x,t) = |t|^{q^+ - 2}t \ln(1 + |t|) + \frac{1}{q^+} \frac{|t|^{q^+ - 1}t}{1 + |t|}.
	\end{align*}
	Its associated primitive function is given by
	\begin{align*}
		F(x,t) = \frac{1}{q^+} |t|^{q^+} \ln(1 + |t|).
	\end{align*}
	It is clear that the nonlinearity $f$ satisfies conditions \eqref{F}\textnormal{(i)}--\textnormal{(iii)}. To verify \textnormal{(iv)}, we compute
	\begin{align*}
		\widetilde{F}(x,t) = \frac{1}{q^+} tf(x,t) - F(x,t) = \frac{1}{(q^+)^2} \frac{|t|^{q^+ + 1}}{1 + |t|},
	\end{align*}
	which is strictly positive for large values of $|t|$. Furthermore, the function $f$ satisfies the inequality \eqref{ex0} for
	\begin{align*}
		\frac{d}{p^-} \leq \sigma \leq \frac{q^+}{q^+ - p^- + 1}.
	\end{align*}
\end{example}

Now, we present the main existence result of this paper.

\begin{theorem} \label{thex}
	Assume that \eqref{H}, \eqref{V0} and \eqref{F} hold. Furthermore, suppose that there exist two positive constants $r, \eta$ satisfying
	\begin{equation}\label{318}
		\max \left\{\eta^{p_{-}}, \eta^{q_{+}}\right\}<\delta r,
	\end{equation}
	such that
	\begin{enumerate}
		\item[$\left(\mathrm{H}_1\right)$]
			$F(x, t) \geq 0$ for a.a.\,$x \in \Omega$ and for all $t \in[0, \eta]$;
		\item[$\left(\mathrm{H}_2\right)$]
			$\alpha(r)<\beta(\eta)$,
	\end{enumerate}
	where $\delta$, $\alpha(r)$, and $\beta(\eta)$ will be defined in Section \ref{App} (see \eqref{thta}, \eqref{ar}, and \eqref{be}, respectively). Then, for each $\lambda \in \Lambda$, where
	\begin{align*}
		\Lambda:=\left[\frac{1}{\beta(\eta)}, \frac{1}{\alpha(r)}\right],
	\end{align*}
	problem \eqref{prb} has at least two nontrivial weak solutions $u_{\lambda, 1}, u_{\lambda, 2} \in \WV$ with opposite energy sign.
\end{theorem}

The condition (iv) in \eqref{F} is crucial for proving that every Cerami sequence is bounded  (see step 1 in Lemma \ref{crmi}). It was initially introduced by Ding--Lee \cite{Ding-Lee-2006} for some scalar Schr\"{o}dinger equations and later slightly modified in Bahrouni--Missaoui--R\u{a}dulescu \cite{Bahrouni-Missaoui-Radulescu-2025} for some nonlinear Robin problems in Orlicz spaces. Note that there are some nonlinearities that fulfill condition (iv) but do not satisfy the following Ambrosetti--Rabinowitz condition (see \cite{Bahrouni-Missaoui-Radulescu-2025})
\begin{enumerate}
	\item[(AR)]
		there exists $\theta > q^+$ such that $0 < \theta F(x, t) \leq tf(x, t)$ for a.a.\,$x \in \RN$ and for all $t \neq 0$.
\end{enumerate}

Recently, Amoroso--Bonanno--D'Agu\`\i--Winkert \cite{Amoroso-Bonanno-DAgui-Winkert-2024} established the existence and multiplicity of bound\-ed solutions for a quasilinear problem driven by the operator \eqref{I4} by using  \cite[Theorem 2.1 and Remark 2.2]{Bonanno-DAgui-2016} which will be also our key tool for the proof of Theorem \ref{thex}. The main difference between their work and ours (aside from the difference in the operators) lies in the assumptions imposed. However, both our assumption \eqref{F}(ii), (iv) and their assumption $(\mathrm{H}_f)(\mathrm{iii}), (\mathrm{iv})$ in \cite{Amoroso-Bonanno-DAgui-Winkert-2024} avoid the need for $f$ to satisfy the usual Ambrosetti--Rabinowitz condition, making them less restrictive. The assumptions \eqref{F} and $(\mathrm{H}_f)$ are similar in that way that both require the primitive of $ f $ to be $q^+$-superlinear. However, instead of the behavior specified by $(\mathrm{H}_f)(\mathrm{iii})$ in \cite{Amoroso-Bonanno-DAgui-Winkert-2024}, we impose a different condition near $\pm \infty$, as described in \eqref{F}(iv). Additionally, we introduce a condition for the behavior of $f$ near the origin, see \eqref{F}(iii). These changes require a new technical analysis to take the Cerami condition into account (see lemma \ref{crmi}).

The paper is organized as follows. In Section \ref{Prel}, we provide the necessary preliminaries, introducing Musielak-Orlicz Sobolev spaces and discussing the functional setting. In particular, the new Musielak-Orlicz Sobolev space $W^{1,\mathcal{H}}(\mathbb{R}^d)$ and the weighted Musielak-Orlicz Sobolev space $W^{1,\mathcal{H}}_V(\mathbb{R}^d)$ are treated in the subsection \ref{FST}. Section \ref{Emb} will be devoted to the proofs of the continuous and compact embedding results, (see Theorems \ref{Inj1}, \ref{Inj2}, \ref{Inj3}, \ref{Inj30} and \ref{thms}). Finally, Section \ref{App} explores applications, including some properties of the energy functional  and the operator related to our problem \eqref{prb}, and  the proof of Theorem \ref{thex}.

\section{Preliminaries}\label{Prel}

This section is divided into two subsections. In the first part, we recall some known definitions and results about Musielak-Orlicz Sobolev spaces. In the second part, we define the new Musielak-Orlicz space $L^{\h}(\RN)$ and the Musielak-Orlicz Sobolev space $W^{1,\h}(\RN)$. Moreover, we introduce the functional space $\WV$ associated with problem \eqref{prb} and establish additional properties of the Sobolev conjugate of $\h$, see Lemmas \ref{lemR2} and \ref{lemR3}.

\subsection{Musielak-Orlicz Sobolev spaces}\label{MOSG}

In this subsection, we explore definitions and properties relevant to the Musielak-Orlicz spaces and Musielak-Orlicz Sobolev spaces. We refer to the paper by Bahrouni--Bahrouni--Missaoui--R\u{a}dulescu \cite{Bahrouni-Bahrouni-Missaoui-Radulescu-2024}, Crespo-Blanco--Gasi\'{n}ski--Harjulehto--Winkert \cite{Crespo-Blanco-Gasinski-Harjulehto-Winkert-2022}, Fan \cite{Fan-2012-1,Fan-2012-2} and the monographs by Chlebicka--Gwiazda--\'{S}wierczewska-Gwiazda \cite{Chlebicka-Gwiazda-Swierczewska-Gwiazda-Wroblewska-Kaminska-2021}, Diening--Harjulehto--H\"{a}st\"{o}--R$\mathring{\text{u}}$\v{z}i\v{c}ka \cite{Diening-Harjulehto-Hasto-Ruzicka-2011}, Harjulehto--H\"{a}st\"{o} \cite{Harjulehto-Hasto-2019}, Musielak \cite{Musielak-1983}, Papageorgiou--Winkert \cite{Papageorgiou-Winkert-2024}.

\begin{definition}\label{dfnfct}
	A function $\h\colon\RN\times[0,\infty)\to [0,\infty)$ is called a generalized \textnormal{N}-function if it satisfies the following conditions:
	\begin{enumerate}
		\item[\textnormal{(1)}]
			For a.a.\,$x\in \RN$, $t\mapsto \h(x,t)$ is even, continuous, nondecreasing and convex, and for all $t\geq 0$, $x\mapsto\h(x,t) $ is measurable;
		\item[\textnormal{(2)}]
			$\displaystyle{\lim\limits_{t\rightarrow 0}\frac{\h(x,t)}{t}=0}$ for a.a.\,$x\in \RN;$
		\item[\textnormal{(3)}]
			$\displaystyle{\lim\limits_{t\rightarrow \infty}\frac{\h(x,t)}{t}=\infty}$ for a.a.\,$x\in \RN;$
		\item[\textnormal{(4)}]
			$\displaystyle{\h(x,t)>0}$ for all $t>0$ and for a.a.\,$x\in \RN$ and $\h(x,0)=0$ for a.a.\,$x\in \RN$.
	\end{enumerate}
\end{definition}

\begin{remark}
	We give an equivalent definition of a generalized \textnormal{N}-function that admits an integral representation, see Pick--Kufner--John--Fu\v{c}\'{\i}k \cite[Definition 4.2.1]{Pick-Kufner-John-Fucik-2013}. For $x \in \RN$ and $t \geq 0$, we denote by $h(x,t)$ the right-hand derivative of $\h(x,\cdot)$ at $t$, and define $h(x,t) =-h(x,-t)$ for $t < 0$. Then for each $x \in \RN$, the function $h(x,\cdot)$ is odd, $h(x,t) \in \R$ for $t \in \R,\ h(x, 0) = 0,\  h(x,t) > 0$ for $t > 0,\ h(x,\cdot) $ is right-continuous and nondecreasing on $[0,+\infty),\  h(x,t)\to +\infty$ as $t \rightarrow +\infty,$ and
	\begin{align*}
		\h(x,t)=\int_{0}^{|t|} h(x,s)\,\mathrm{d}s\quad \text{ for } x\in \RN \text{ and } t \in \R.
	\end{align*}
\end{remark}

\begin{definition}$~$
	\begin{enumerate}
		\item[\textnormal{(1)}]
			We say that a generalized \textnormal{N}-function $\h$ satisfies the $\Delta_2$-condition if there exist $C_0 > 0$ and a nonnegative function $m \in L^1(\RN)$ such that
			\begin{align*}
				\h(x,2t)\leq C_0\h(x,t)+m(x)\quad\text{for a.a.\,}x\in \RN \text{ and for all } t\geq 0.
			\end{align*}
		\item[\textnormal{(2)}]
			A generalized \textnormal{N}-function $\h$ is said to satisfy the condition:
			\begin{enumerate}
				\item[\textnormal{(A$_0$)}]
					if there exists $\beta \in(0,1]$ such that
					\begin{align*}
						\beta \leq \varphi^{-1}(x, 1) \leq \frac{1}{\beta}
					\end{align*}
					for a.a.\,$x \in \mathbb{R}^d$;
				\item[\textnormal{(A$_1$)}]
					if there exists $\beta \in(0,1]$ such that
					\begin{align*}
						\beta \varphi^{-1}(x, t) \leq \varphi^{-1}(y, t)
					\end{align*}
					for every $t \in\left[1, \frac{1}{|B|}\right]$, for a.a.\,$x, y \in B$, and every ball $B \subset \mathbb{R}^d$ with $|B| \leq 1$;
				\item[\textnormal{(A$_2$)}]
					if for every $s>0$ there exist $\beta \in(0,1]$ and $h \in L^1\left(\mathbb{R}^n\right) \cap L^{\infty}\left(\mathbb{R}^n\right)$ such that
					\begin{align*}
						\beta \varphi^{-1}(x, t) \leq \varphi^{-1}(y, t)
					\end{align*}
					for a.a.\,$x, y \in \mathbb{R}^d$ and for all $t \in[h(x)+h(y), s]$.
			\end{enumerate}
	\end{enumerate}
\end{definition}

\begin{definition}\label{ddffd}
	Let $\h_1$ and $\h_2$ be two generalized \textnormal{N}-functions.
	\begin{enumerate}
		\item[\textnormal{(1)}]
			We say that $\h_1$ increases essentially slower than $\h_2$ near infinity and we write $\h_1 \ll \h_2$, if for any $k>0$
			\begin{align*}
				\lim _{t \rightarrow \infty} \frac{\h_1(x, k t)}{\h_2(x, t)}=0\quad \text {uniformly in}\  x \in \RN.
			\end{align*}
		\item[\textnormal{(2)}]
			We say that $\h_1 $ is weaker than $\h_2$, denoted by $\h_1 \preceq \h_2$, if there exist two positive constants $C_1, C_2$ and a nonnegative function $m \in L^1(\RN)$ such that
			\begin{align*}
				\h_1(x, t) \leq C_1 \h_2\left(x, C_2 t\right)+m(x)\quad \text{for a.a.\,} x\in \RN \text{ and for all }  t\geq 0.
			\end{align*}
	\end{enumerate}
\end{definition}

\begin{definition}
	For any generalized \textnormal{N}-function $\h$, the function $\widetilde{\h}\colon\RN\times\mathbb{R}\to \mathbb{R}$ defined by
	\begin{align*}
		\widetilde{\h}(x,t):=\sup_{\tau\geq 0}\left( t\tau-\h(x,\tau)\right)\quad \text{for all } x\in \RN \text{ and for all }  t\geq 0,
	\end{align*}
	is called the conjugate function of $\h$.
\end{definition}

In view of the definition of the conjugate function $\widetilde{\h}$, we have the following Young type inequality:
\begin{align}\label{Yi}
	\tau\sigma\leq \h(x,\tau)+\widetilde{\h}(x,\sigma)\quad \text{for all } x\in \RN \text{ and for all }  \tau,\sigma\geq0.
\end{align}

\begin{remark}
	$~$
	\begin{enumerate}
		\item[\textnormal{(1)}]
			Note that the conjugate function $\widetilde{\h}$ is also a generalized \textnormal{N}-function.
		\item[\textnormal{(2)}]
			If there exist $m,\ell\in \mathbb{R}$ such that
			\begin{align}\label{D2}
				1\leq m\leq \frac{h(x,t)t}{\h(x,t)}\leq \ell\quad \text{for all } x\in \RN \text{ and for all }  t> 0,
			\end{align}
			then the generalized \textnormal{N}-function $\h$ satisfies the $\Delta_2$-condition, see \cite{Chlebicka-Gwiazda-Swierczewska-Gwiazda-Wroblewska-Kaminska-2021}.
	\end{enumerate}
\end{remark}

The next lemma is taken from Bahrouni--Bahrouni--Missaoui--R\u{a}dulescu \cite[Lemma 2.3]{Bahrouni-Bahrouni-Missaoui-Radulescu-2024}.

\begin{lemma}\label{lm1}
	Let $\h$ be a generalized \textnormal{N}-function. We suppose that $t\mapsto h(x,t)$ is continuous and nondecreasing for a.a.\,$x\in \RN$. Moreover, we assume that there exist $m,\ell\in \mathbb{R}$ such that
	\begin{align}\label{D22}
		1<m\leq \frac{h(x,t)t}{\h(x,t)}\leq \ell\quad \text{for all } x\in \RN \text{ and for all }  t> 0,
	\end{align}
	then,
	\begin{align*}
		\widetilde{\h}(x,h(x,s))\leq (\ell-1)\h(x,s)\quad \text{for all } s\geq 0 \text{ and for all }  x\in \RN,
	\end{align*}
	and
	\begin{align*}
		\frac{\ell}{\ell-1}:=\widetilde{m}\leq \frac{\widetilde{h}(x,s)s}{\widetilde{\h}(x,s)}\leq \widetilde{\ell}:=\frac{m}{m-1}\quad \text{for all } x\in \RN \text{ and for all }  s> 0,
	\end{align*}
	where $\widetilde{\h}(x,s)=\int_{0}^s \widetilde{h}(x,\upsilon)\,\mathrm{d}\upsilon$.
\end{lemma}

\begin{remark}\label{compl}
	Note that the condition \eqref{D22} implies that $\h$ and its conjugate function $\widetilde{\h}$ satisfy the $\Delta_2$-condition.
\end{remark}

Now, we can define the Musielak-Orlicz space. For $M(\R^d)$ being the set of all measurable functions $u\colon\R^d\to\R$, we define
\begin{align*}
	L^{\h}(\RN):=\left\lbrace u\in M(\R^d)\colon  \rho_{\h}(\lambda u)<+\infty \text{ for some}\ \lambda>0\right\rbrace,
\end{align*}
where
\begin{align}\label{Mo}
	\rho_{\h}(u):= \int_{\RN}\h(x, u)\,\mathrm{d}x.
\end{align}
The space $L^{\h}(\RN)$ is endowed with the Luxemburg norm
\begin{align*}
	\Vert u\Vert_{\lh}:=\inf\left\lbrace \lambda>0\colon  \rho_{\h}\left(\frac{u}{\lambda}\right)\leq 1\right\rbrace.
\end{align*}

\begin{proposition}
	Let $\h$ be a generalized \textnormal{N}-function that satisfies the $\Delta_2$-condition, then
	\begin{align*}
		\lh =\left\lbrace  u\in M(\R^d)\colon  \rho_{\h}( u)<+\infty\right\rbrace.
	\end{align*}
\end{proposition}

\begin{proposition} \label{zoo}
	Let $\h$ be a generalized \textnormal{N}-function fulfilling \eqref{D22}, then the following assertions hold:
	\begin{enumerate}
		\item[\textnormal{(1)}]
			$\min \{\l^m, \l^\ell\}\h(x,t)\leq  \h(x,\l t)\leq \max \{ \l^m, \l^\ell\}\h(x,t)$ for a.a.\,$x \in \RN$ and for all $\l, \ t \geq 0. $
		\item[\textnormal{(2)}]
			$\min \left\{\|u\|_{\lh}^{m},\|u\|_{\lh}^{\ell}\right\} \leq \rho_{\mathcal{H}}(u) \leq\max \left\{\|u\|_{\lh}^{m},\|u\|_{\lh}^{\ell}\right\}$ for all $  u\in \lh$.
		\item[\textnormal{(3)}]
			Let $\left\{ u_n\right\}_{n \in \mathbb{N}}\subseteq \lh $ and $  u \in \lh $, then
			\begin{align*}
				\|u_n-u\|_{\lh }\to0 \quad \Longleftrightarrow \quad \rho_\h (u_n-u) \to 0\quad \text{as } n \rightarrow +\infty.
			\end{align*}
	\end{enumerate}
\end{proposition}

As a consequence of \eqref{Yi}, we have the following result.

\begin{lemma}[H\"older's type inequality]\label{H1}
	Let  $\h$ be a generalized \textnormal{N}-function that satisfies \eqref{D22}, then
	\begin{align*}
		\left\vert \int_{\RN} uv\,\mathrm{d}x \right\vert \leq 2 \Vert u\Vert_{\lh}\Vert v\Vert_{L^{\widetilde{\h}}(\RN)}\quad \text{for all } u\in L^{\h}(\RN) \text{ and for all } v\in L^{\widetilde{\h}}(\RN).
	\end{align*}
\end{lemma}

The subsequent proposition deals with some topological properties of the Musielak-Orlicz space, see Musielak \cite[Theorem 7.7 and Theorem 8.5]{Musielak-1983}.

\begin{proposition}\label{AB}
	Let $\h$ be a generalized \textnormal{N}-function.
	\begin{enumerate}
		\item[\textnormal{(i)}]
			The space $\left(L^{\h}(\RN),\|\cdot\|_{\lh}\right)$ is a Banach space.
		\item[\textnormal{(ii)}]
			If $\h$ satisfies \eqref{D2}, then $L^{\h}(\RN)$ is a separable space.
		\item[\textnormal{(iii)}]
			If $\h$ satisfies \eqref{D22}, then $L^{\h}(\RN)$ is a reflexive space.
	\end{enumerate}
\end{proposition}

Now, we are ready to define the Musielak-Orlicz Sobolev space. Let $\h$ be a generalized \textnormal{N}-function. The Musielak-Orlicz Sobolev space is defined as follows
\begin{align*}
	W^{1,\h}(\RN):=\left\lbrace u\in L^{\h}(\RN)\colon  |\nabla u| \in L^{\h}(\RN)\right\rbrace.
\end{align*}
The space $W^{1,\h}(\RN)$ is endowed with the norm
\begin{align*}
	\Vert u\Vert_{\W}:=\Vert u\Vert_{L^{\h}(\RN)}+\Vert \nabla u\Vert_{L^{\h}(\RN)}\quad \text{for all } u\in W^{1,\h}(\RN),
\end{align*}
where $\|\nabla u\|_{L^{\h}(\RN)} := \|\, |\nabla u|\, \|_{L^{\h}(\RN)}$.

\begin{remark}
	If $\h$  satisfies \eqref{D22}, then the space $W^{1,\h}(\RN)$ is a reflexive and separable Banach space with respect to the norm $\Vert \cdot\Vert_{\ww}$.
\end{remark}

\subsection{Functional setting} \label{FST}

In this subsection, we turn our attention to the specific generalized \textnormal{N}-function $\h$, see \eqref{hdfn} and \eqref{h}. It was developed to solve problems involving a double phase operator with variable exponents, such as in problem \eqref{prb}, where the exponents depend on both the unknown solution and its gradient. Based on the previous subsection, we establish key properties of $\h$, its conjugate function $\widetilde{\h}$, and its associated Sobolev conjugate $\h_\ast$. We then introduce the space $\W$ and summarize its essential properties, see the work of Bahrouni--Bahrouni--Missaoui--R\u{a}dulescu \cite{Bahrouni-Bahrouni-Missaoui-Radulescu-2024}. In the last part of this subsection, we focus on the function space $\WV$ which is relevant to our problem, and we rigorously prove its basic properties.

To this end, we introduce the function $\h\colon\RN \times [0,+\infty[ \rightarrow [0,+\infty[$ by
\begin{align*}
	\h(x,t)&= \int_{0}^{t} h(x,s)\,\mathrm{d}s = \int_{0}^{t}\L( s^{p(x,s)-1}+\mu (x) s^{q(x,s)-2}\r)\,\mathrm{d}s,
\end{align*}
where $0\leq \mu(\cdot) \in C^{0,1}(\mathbb{R}^d)$ and $p(\cdot, \cdot)$ as well as $q(\cdot, \cdot)$ satisfy \eqref{H}.

We start with the following proposition, which is proven by Bahrouni--Bahrouni--Missaoui--R\u{a}dulescu \cite[Propositions 3.1 and 3.3]{Bahrouni-Bahrouni-Missaoui-Radulescu-2024}.

\begin{proposition} \label{nfct}
	Let hypotheses \eqref{H} \textnormal{(i)}, \textnormal{(ii)} be satisfied. Then following hold:
	\begin{enumerate}
		\item[\textnormal{(i)}]
			$\h$ is a generalized \textnormal{N}-function.
		\item[\textnormal{(ii)}]
			$\h$ and $\widetilde{\h}$ fulfill the $\Delta_2$-condition, and we have that
			\begin{align}\label{l22}
				p^{-} \leq \frac{ h(x, t)t}{\mathcal{H}(x, t)} \leq q^+\quad \text{for all } t > 0 \text{ and for a.a.\,}x \in \RN.
			\end{align}
	\end{enumerate}
\end{proposition}

The following proposition plays a crucial role in the proof of the continuous embedding theorem related to $\W$.

\begin{proposition}\label{Prop1}
	Let hypotheses \eqref{H} \textnormal{(i)}--\textnormal{(iv)} be satisfied. Then, $\mathcal{H}$ satisfies \ \textnormal{(A$_0$)}, \textnormal{(A$_1$)} and \textnormal{(A$_2$)}.
\end{proposition}

\begin{proof}
	For conditions \textnormal{(A$_0$)} and \textnormal{(A$_1$)}, the proofs are similar to those done by Bahrouni--Bahrouni--Missaoui--R\u{a}dulescu \cite[Proposition 3.10]{Bahrouni-Bahrouni-Missaoui-Radulescu-2024}. We only need to prove that $\h$ satisfies \textnormal{(A$_2$)}. To this end, for $t\in[0,s]$ and from \eqref{H} (ii) and Proposition \ref{nfct}, we have
	\begin{align*}
		\h (x,t) \approx t^{p(x)},
	\end{align*}
	where "$\approx$" indicates that there exist positive constants $c_1$ and $c_2$ depending on $s$ such that $c_1t^{p^-} \leq \h(x,t) \leq c_2t^{p^-}$ for a.a.\,$x \in \RN$ and for all $0\leq t \leq s$. Hence, the condition \textnormal{(A$_2$)} follows from Lemma 4.2.5 by   Harjulehto--H\"{a}st\"{o} \cite[Lemma 4.2.5]{Harjulehto-Hasto-2019}.
\end{proof}

Thanks to Proposition \ref{nfct}, we are ready to define the  Musielak–Orlicz space $L^{\h}(\RN)$ by
\begin{align*}
	L^{\h}(\RN)=\curly{u\in M(\RN)\colon \vi(u)<+\infty },
\end{align*}
endowed with the following norm
\begin{align*}
	\|u\|_{L^{\h}(\RN)} =\inf \bigg\{ \l >0\colon  \vi\left(\frac{u}{\l}\right) \leq1\bigg\},
\end{align*}
where $\vi$ is defined as in \eqref{Mo}. Proceeding similarly to Subsection \ref{MOSG}, we can introduce the space $\W$ equipped with the norm
\begin{align}\label{nhr}
	\|u\|_{\W}:=\|u\|_{L^{\h}(\RN)}+\|\nabla u\|_{L^{\h}(\RN)}.
\end{align}

\begin{remark}\label{mark}
	According to Proposition \ref{nfct}, the results presented in the previous subsection remain valid with our special \textnormal{N}-function $\h$. This includes Proposition \ref{zoo} with $m = p^-$ and $\ell = q^+$ as well as Lemma \ref{H1} and Proposition \ref{AB}.
\end{remark}

Based on the work by de Albuquerque--de Assis--Carvalho--Salort \cite{deAlbuquerque-Assis-Carvalho-Salort-2023}, we present an adaptation of the classical Br\'{e}zis-Lieb Lemma (see Br\'{e}zis--Lieb \cite{Brezis-Lieb-1983}) for modular functions associated with our function $\h$.

\begin{proposition}[Br\'{e}zis-Lieb Lemma in Musielak-Orlicz Spaces]\label{b-l}
	Assume that \eqref{H} \textnormal{(i)}, \textnormal{(ii)} hold. Let $\{u_n\}_{n \in \N}$ be a sequence in $L^{\h}(\mathbb{R}^d)$ such that $u_n \rightharpoonup u$ weakly in $L^{\h}(\mathbb{R}^d)$. Then,
	\begin{align*}
		\lim_{n \to \infty} \left( \rho_{\h}(u_n) - \rho_{h}(u_n - u) - \rho_{\h}(u) \right) = 0.
	\end{align*}
\end{proposition}

\begin{lemma} \label{Aux1}
	Let \eqref{H} \textnormal{(i)}, \textnormal{(ii)} be satisfied and $B \subset \mathbb{R}^d$ be a measurable set with $\left\vert B \right|<\infty$. Then, there exist $C_1,\, C_2>0 $ such that
	\begin{align*}
		C_1 \min \left\{ |B|^{\frac{1}{p^-}}, \ |B|^{\frac{1}{q^+}} \right\} \leq \|\chi_B\|_{\Lh} \leq C_2 \max \left\{ |B|^{\frac{1}{p^-}}, \ |B|^{\frac{1}{q^+}} \right\}.
	\end{align*}
\end{lemma}

\begin{proof}
	Using Proposition \ref{zoo} and Remark \ref{mark}, we infer that
	\begin{equation}\label{2eq30}
		\min \left\lbrace \left\Vert \chi_{_B} \right\Vert_{L^{\h}(\mathbb{R}^d)}^{p^-}, \left\Vert \chi_{_B} \right\Vert_{L^{\h}(\mathbb{R}^d)}^{q^+} \right\rbrace \leq \int_{\mathbb{R}^d} \h(x, \chi_{_B}(x)) \,\mathrm{d}x \leq \max \left\lbrace \left\Vert \chi_{_B} \right\Vert_{L^{\h}(\mathbb{R}^d)}^{p^-}, \left\Vert \chi_{_B} \right\Vert_{L^{\h}(\mathbb{R}^d)}^{q^+} \right\rbrace.
	\end{equation}
	Furthermore, using assumption \eqref{H} (i), we find that
	\begin{equation}\label{2eq31}
		\frac{1}{p^-} \left| B \right| \leq \int_{\mathbb{R}^d} \h(x, \chi_{_B}(t)) \,\mathrm{d}x = \int_{B} \h(x, 1) \,\mathrm{d}x \leq \frac{1}{p^-} (1 + \left|\mu\right|_{L^\infty(\mathbb{R}^d)}) \left| B \right|.
	\end{equation}
	By combining \eqref{2eq30} and \eqref{2eq31}, we deduce that
	\begin{align*}
		\min \left\lbrace \left( C_1 \left| B \right| \right)^{\frac{1}{p^-}}, \left( C_1 \left| B \right| \right)^{\frac{1}{q^+}} \right\rbrace \leq \left\Vert \chi_{_B} \right\Vert_{L^{\h}(\mathbb{R}^d)} \leq \max \left\lbrace \left( C_2 \left| B \right| \right)^{\frac{1}{p^-}}, \left( C_2 \left| B \right| \right)^{\frac{1}{q^+}} \right\rbrace,
	\end{align*}
	where $ C_1 = \frac{1}{p^-} $ and $ C_2 = \frac{1}{p^-} (1 + \|\mu\|_{\infty}) $. This concludes the proof.
\end{proof}

\begin{definition}\label{csh}$~$
	\begin{enumerate}
		\item[\textnormal{(i)}]
			The Sobolev conjugate of $\mathcal{H}$ is defined as the generalized \textnormal{N}-function $\mathcal{H}_\ast$ given by
			\begin{align*}
				\mathcal{H}_\ast(x, t):= \overline{\mathcal{H}}\left(x, \mathcal{N}^{-1}(x, t)\right) \quad \text{for } x \in \RN \text { and } t \geq 0,
			\end{align*}
			where $\mathcal{N}\colon \RN \times[0, +\infty) \rightarrow[0,+ \infty)$ is the function defined by
			\begin{align}\label{eqN}
				\mathcal{N}(x, t):=\left(\int_0^t\left(\frac{\tau}{\overline{\mathcal{H}}(x, \tau)}\right)^{\frac{1}{d-1}} d \tau\right)^{\frac{1}{d^{\prime}}}\quad\text{for } x \in \RN \text { and } t \geq 0,
			\end{align}
			with $d':=\frac{d}{d-1}$ and
			\begin{align*}
				\overline{\mathcal{H}}(x,t):=
				\begin{cases}
					2\max\left\{ \h\left(x, \h^{-1}(x,1)t\right), 2t-1 \right\}-1 & \text{if } t \geq 1\\[1ex]
					\ds\limsup_{\abs{x}\rightarrow +\infty} \max\left\{ \h\left(x, \h^{-1}(x,1)t\right), 2t-1 \right\} & \text{if } 0\leq t <1,
				\end{cases}
			\end{align*}
			for all $x \in \RN$.
		\item[\textnormal{(ii)}]
			Assume that  $\varphi$ is a generalized \textnormal{N}-function. The homogeneous Musielak-Orlicz Sobolev space $V^{1, \varphi}(\O)$ is defined as
			\begin{align*}
				V^{1, \varphi}(\RN)=\left\{u \in W_{\operatorname{loc}}^{1,1}(\RN)\colon |\nabla u| \in L^{\varphi}(\RN)\right\}.
			\end{align*}
			Let $\O$ be a bounded open set such that $\overline{\O} \subset \RN$, then
			\begin{align*}
				\|u\|_{L^1(\O)}+\|\nabla u\|_{L^{\varphi}(\RN)}
			\end{align*}
			defines a norm on $V^{1, \varphi}(\RN)$.
		\item[\textnormal{(iii)}]
			We also define the space
			\begin{align*}
				V^{1, \varphi}_{\textrm{d}}(\RN)=\curly{u \in V^{1, \varphi}(\RN)\colon \abs {\curly{| u|>t}}<\infty \text{ for every } t > 0},
			\end{align*}
			where $| \cdot|$ is the Lebesgue measure and it can equipped with the norm
			\begin{align*}
				\norm{\nabla u}_{L^\varphi(\RN)}.
			\end{align*}
	\end{enumerate}
\end{definition}

\begin{remark}\label{RKN}
	By Propositions \ref{nfct} and \ref{Prop1}, we know that $\h$ satisfies the $\Delta_2$-condition and \textnormal{(A$_0$)}. Therefore, invoking the recent work by Cianchi--Diening \cite{Cianchi-Diening-2024}, the function $\overline{\mathcal{H}}$ is equivalent, in the sense of the relation $\approx$, to the function
	\begin{align*}
		\h^\circ(x,t):=
		\begin{cases}
			\h(x,t) & \text{if } t \geq 1\\[1ex]
			\h_\infty(t):= \ds \limsup_{\abs{x}\rightarrow +\infty} \h(x,t)=:\frac{t^{p_\infty}}{p_\infty}+\frac{\mu_\infty}{q_\infty}t^{q_\infty} & \text{if } 0\leq t <1,
		\end{cases}
	\end{align*}
	for all $x \in \RN$, where $p_\infty$, $q_\infty$ and $\mu_\infty$ represent the supremum limits of $p(x)$, $q(x)$, and $\mu(x)$, respectively, as $|x| \rightarrow +\infty$. Thus, by \cite[Remark 3.2]{Cianchi-Diening-2024}, we can replace $ \overline{\mathcal{H}} $ with $ \h^\circ $ in definition of $\h_\ast$. This equivalence and substitution are crucial for simplifying our subsequent analysis and ensuring consistency within the framework of our definitions and propositions.
\end{remark}

\begin{proposition}\label{PRPN}
	Assume hypotheses \eqref{H}\textnormal{(i)} and \textnormal{(ii)}. Then the following hold:
	\begin{enumerate}
		\item[\textnormal{(1)}]
			$\min \left\lbrace \xi^{p^-}, \xi^{q^+} \right\rbrace \h^\circ (x,t) \leq \h^\circ (x, \xi t) \leq \max \left\lbrace \xi^{p^-}, \xi^{q^+} \right\rbrace \h^\circ (x,t)$.
		\item[\textnormal{(2)}]
			$p^{-} \leq \ds\frac{ h^\circ(x, t)t}{\mathcal{H}^\circ(x, t)} \leq q^+$ for all $t > 0$ and for a.a.\,$x \in \mathbb{R}^d$, where
			\begin{align*}
				h^\circ(x,t) =
				\begin{cases}
					h(x,t), & \text{for } t \geq 1, \\
					\h'_\infty (t), & \text{if } 0 \leq t < 1.
				\end{cases}
			\end{align*}
	\end{enumerate}
\end{proposition}

\begin{proof}
	Clearly, $\h^\circ(x,t)=\h(x,t)$ for a.a.\,$x \in \RN$ and for all $t \geq 1$. Therefore, for $t\geq 1,$ conditions (1) and (2) follow from Propositions \ref{zoo} and \ref{nfct}. For $0 \leq t \leq 1$, consider the function $G\colon \RN \times \R \rightarrow \R$ defined by
	\begin{align*}
		G(x,t)=\frac{t^{p_\infty}}{p_\infty}+\frac{\mu_\infty}{q_\infty}t^{q_\infty}.
	\end{align*}
	Note that $G$ is a generalized \textnormal{N}-function that satisfies \eqref{D22} with $m=p^-$ and $\ell=q^+$, and therefore satisfies the conditions of Propositions \ref{zoo} and \ref{nfct}. Thus, since $\ds\h_\infty = G_{|_{[0,1]}}$ for $0 \leq t \leq 1$, we can conclude the proof.
\end{proof}

Now, we require the following two lemmas.

\begin{lemma}\label{lemR2}
	Under the assumptions \eqref{H}\textnormal{(i)}, \textnormal{(ii)}, the following inequality holds:
	\begin{align*}
		\min\left\lbrace t^{p^-_{*}}, t^{q^+_{*}} \right\rbrace \h_*(x,\xi) \leq \h_*(x,t\xi) \leq \max\left\lbrace t^{p^-_{*}}, t^{q^+_{*}} \right\rbrace \h_*(x,\xi), \quad \text{for all } t, \xi \geq 0,
	\end{align*}
	where $\displaystyle{p_*^- := \frac{dp^-}{d - p^-}}$ and $\displaystyle{q_*^+ := \frac{dq^+}{d - q^+}}$.
\end{lemma}

\begin{proof}
	By the definition of $\mathcal{N}$ (see \eqref{eqN} and Remark \ref{RKN}), for all $t > 0$ and $\xi \geq 0$, we have
	\begin{align*}
		\mathcal{N}(x,t\xi) = \left( \int_0^{t\xi} \left( \frac{s}{\h^\circ(x,s)} \right)^{\frac{1}{d-1}} \,\mathrm{d}s \right)^{\frac{d-1}{d}} = t^{\frac{d-1}{d}} \left( \int_0^\xi \left( \frac{ts}{\h^\circ(x,ts)} \right)^{\frac{1}{d-1}} \,\mathrm{d}s \right)^{\frac{d-1}{d}}.
	\end{align*}
	Using Proposition \ref{PRPN}-(1), for all $0 < t \leq 1$ and $\xi \geq 0$, we get
	\begin{align*}
		\mathcal{N}(x,t\xi)
		&\leq t^{\frac{d-1}{d}} \left( \int_0^\xi \left( \frac{ts}{t^{q^+} \h^\circ(x,s)} \right)^{\frac{1}{d-1}} \,\mathrm{d}s \right)^{\frac{d-1}{d}} = t^{\left( \frac{d-1}{d} - \frac{ q^+ - 1}{d} \right)} \left( \int_0^\xi \left( \frac{s}{\h^\circ(x,s)} \right)^{\frac{1}{d-1}} \,\mathrm{d}s \right)^{\frac{d-1}{d}} \\
		& = t^{\frac{d - q^+}{d}} \mathcal{N}(x,\xi)
	\end{align*}
	and
	\begin{align*}
		\mathcal{N}(x,t\xi)
		&\geq t^{\frac{d-1}{d}} \left( \int_0^\xi \left( \frac{ts}{t^{p^-} \h^\circ(x,s)} \right)^{\frac{1}{d-1}} \,\mathrm{d}s \right)^{\frac{d-1}{d}} = t^{\left( \frac{d-1}{d} - \frac{ p^- - 1}{d} \right)} \left( \int_0^\xi \left( \frac{s}{\h^\circ(x,s)} \right)^{\frac{1}{d-1}} \,\mathrm{d}s \right)^{\frac{d-1}{d}} \\
		& = t^{\frac{d - p^-}{d}} \mathcal{N}(x,\xi).
	\end{align*}
	Thus,
	\begin{align*}
		t^{\frac{d - p^-}{d}} \mathcal{N}(x,\xi) \leq \mathcal{N}(x,t\xi) \leq t^{\frac{d - q^+}{d}} \mathcal{N}(x,\xi), \quad \text{for all } 0 \leq t \leq 1 \text{ and all } \xi \geq 0.
	\end{align*}
	Similarly, we have
	\begin{align*}
		t^{\frac{d - q^+}{d}} \mathcal{N}(x,\xi) \leq \mathcal{N}(x,t\xi) \leq t^{\frac{d - p^-}{d}} \mathcal{N}(x,\xi), \quad \text{for all } t > 1 \text{ and all } \xi \geq 0.
	\end{align*}
	Combining these results, we obtain
	\begin{equation}\label{aml1}
		\zeta_0(t) \mathcal{N}(x,\xi) \leq \mathcal{N}(x,t\xi) \leq \zeta_1(t) \mathcal{N}(x,\xi), \quad \text{for all } t, \xi \geq 0,
	\end{equation}
	where
	\begin{align*}
		\zeta_0(t) = \min\left\lbrace t^{\frac{d - p^-}{d}}, t^{\frac{d - q^+}{d}} \right\rbrace \quad \text{and} \quad \zeta_1(t) = \max\left\lbrace t^{\frac{d - p^-}{d}}, t^{\frac{d - q^+}{d}} \right\rbrace.
	\end{align*}
	Therefore, inserting $\tau = \mathcal{N}(x,\xi)$ and $\kappa = \zeta_0(t)$ into the inequality \eqref{aml1}, i.e., $\xi = \mathcal{N}^{-1}(x,\tau)$ and $t = \zeta_0^{-1}(\kappa)$, we get
	\begin{align*}
		\kappa \tau \leq \mathcal{N}(x,\zeta_0^{-1}(\kappa)\mathcal{N}^{-1}(x,\tau)).
	\end{align*}
	Since $\mathcal{N}^{-1}$ is non-decreasing, we infer that
	\begin{align*}
		\mathcal{N}^{-1}(x,\kappa \tau) \leq \zeta_0^{-1}(\kappa) \mathcal{N}^{-1}(x,\tau), \quad \text{for all } \kappa, \tau > 0.
	\end{align*}
	Similarly, putting $\tau = \mathcal{N}(x,\xi)$ and $\kappa = \zeta_1(t)$ into the inequality \eqref{aml1}, we obtain
	\begin{align*}
		\zeta_1^{-1}(\kappa) \mathcal{N}^{-1}(x,\tau) \leq \mathcal{N}^{-1}(x,\kappa \tau), \quad \text{for all } \kappa, \tau > 0.
	\end{align*}
	From these results, it follows that
	\begin{align*}
		\min\left\lbrace t^{\frac{d}{d - p^-}}, t^{\frac{d}{d - q^+}} \right\rbrace \mathcal{N}^{-1}(x,\xi) \leq \mathcal{N}^{-1}(x,t\xi) \leq \max\left\lbrace t^{\frac{d}{d - p^-}}, t^{\frac{d}{d - q^+}} \right\rbrace \mathcal{N}^{-1}(x,\xi), \quad \text{for all } t, \xi \geq 0.
	\end{align*}
	It follows, from Proposition \ref{PRPN}-(1), that
	\begin{align*}
		\min\left\lbrace t^{p^-_{*}}, t^{q^+_{*}} \right\rbrace \h_*(x,\xi) \leq \h_*(x,t\xi) \leq \max\left\lbrace t^{p^-_{*}}, t^{q^+_{*}} \right\rbrace \h_*(x,\xi), \quad \text{for all } t, \xi \geq 0,
	\end{align*}
	where $\displaystyle{p_*^- = \frac{dp^-}{d - p^-}}$ and $\displaystyle{q_*^+ = \frac{dq^+}{d - q^+}}$. This completes the proof.
\end{proof}

\begin{lemma}\label{lemR3}
	Let hypotheses \eqref{H}\textnormal{(i)}, \textnormal{(ii)}, be satisfied.
	\begin{enumerate}
		\item[\textnormal{(i)}]
			It holds
			\begin{align*}
				p_*^- \leq \frac{h_*(x,t) t}{\h_*(x,t)} \leq q_*^+, \quad \text{for all } t > 0,
			\end{align*}
			where $\h_*(x,t) = \ds\int_0^t h_*(x,s) \,\mathrm{d}s$.
		\item[\textnormal{(ii)}]
			The function $\h_*$ satisfies the $\Delta_2$-condition and it holds
			\begin{align*}
				\min\left\lbrace \Vert u \Vert_{L^{\h_*}(\mathbb{R}^d)}^{p^-_*}, \Vert u \Vert_{L^{\h_*}(\mathbb{R}^d)}^{q^+_*} \right\rbrace \leq \int_{\RN} \h_*(x,u)\,\mathrm{d}x \leq \max\left\lbrace \Vert u \Vert_{L^{\h_*}(\mathbb{R}^d)}^{p^-_*}, \Vert u \Vert_{L^{\h_*}(\mathbb{R}^d)}^{q^+_*} \right\rbrace,
			\end{align*}
			for all $u \in L^{\h_*}(\RN)$.
	\end{enumerate}
\end{lemma}

\begin{proof}
	(i) From Lemma \ref{lemR2}, we have
	\begin{align*}
		t^{p_*^-} \h_*(x,z) \leq \h_*(x,tz) \leq t^{q_*^+} \h_*(x,z) \quad \text{for all } t > 1 \text{ and for all } z>0,
	\end{align*}
	which implies that
	\begin{align*}
		\frac{t^{p_*^-} - 1^{p^-_*}}{t - 1} \h_*(x,z) \leq \frac{\h_*(x,tz) - \h_*(x,z)}{t - 1} \leq \frac{t^{q_*^+} - 1^{q^+_*}}{t - 1} \h_*(x,z)
	\end{align*}
	for all $t > 1$ and for all $z>0$. Taking the limit as $t \to 1$, we deduce that
	\begin{align*}
		p_*^- \leq \frac{h_*(x,z) z}{\h_*(x,z)} \leq q_*^+ \quad \text{for all } z > 0.
	\end{align*}
	(ii) Assertion (ii) follows directly from (i) and Remark \ref{compl}. This concludes the proof.
\end{proof}

Using an argument similar to the one presented in the proof of Lemma 2.14 by Bahrouni--Missaoui--Ounaies \cite{Bahrouni-Missaoui-Ounaies-2024}, we derive the following lemma.

\begin{lemma}\label{Aux}
	Let hypotheses \eqref{H}\textnormal{(i)}, \textnormal{(ii)} be satisfied. Then, there exists a generalized \textnormal{N}-function $\mathcal{R}$ that satisfies the following conditions:
	\begin{enumerate}
		\item[\textnormal{(1)}]
			$1 < r^- \leq \dfrac{r(x,t)t}{\mathcal{R}(x,t)} \leq r^+ < \dfrac{p^-_*}{q^+}$ for a.a.\,$x \in \mathbb{R}^d$ and for all $t \geq 0$.
		\item[\textnormal{(2)}]
			There exist constants $c_1, c_2 > 0$ such that
			\begin{align*}
				c_1 < \mathcal{R}(x,1) < c_2 \quad \text{for a.a.\,} x \in \mathbb{R}^d.
			\end{align*}
		\item[\textnormal{(3)}]
			$\mathcal{R} \circ \h \ll \h_*$.
	\end{enumerate}
\end{lemma}

In the following paragraph, we focus on the weighted functional space associated with our problem \eqref{prb} where $V \not\equiv 1$. We consider the function $\h_V$ defined as
\begin{align*}
	\h_V \colon \RN\times [0,+\infty[\to[0,+\infty[, \quad
	(x,t) \mapsto V(x) \h (x,t).
\end{align*}

For the rest of this part, we assume that $V$  fulfills the assumption \eqref{V0}(i). Invoking the hypothesis \eqref{V0}(i) and Proposition \ref{nfct}, it is easy to prove that $\h_V $ is a generalized \textnormal{N}-function and satisfies the $\Delta_2$-condition under the assumption \eqref{H}(i), (ii). By this fact we can define the weighted Musielak-Orlicz space $\L(  L^{\h}_V(\RN), \|\cdot\|_{L^{\h}_V(\RN)}\r)$ as
\begin{align*}
	L^{\h}_V(\RN):=\curly{ u\in M(\RN)\colon  \int_{\RN}^{}\h_V \L(x,|u|\r)\,\mathrm{d}x <+\infty},
\end{align*}
endowed with the  Luxemburg norm
\begin{align*}
	\|u\|_{L^{\h}_V(\RN)}:= \inf \curly{\l>0\colon \int_{\RN}^{}\h_V \L(x,\frac{|u|}{\l}\r)\,\mathrm{d}x \leq1}.
\end{align*}

\begin{remark}
	Clearly, the function $\h_V$ satisfies the inequality \eqref{D22} with $m = p^-$ and $\ell = q^+$. Therefore, the results of Proposition \ref{AB} regarding the space $L^{\h_V}(\mathbb{R}^N)$ are valid. Moreover, the relationships between the norm in $L^{\h_V}(\mathbb{R}^N)$ and the modular associated with $\h_V$ still hold, as described in Proposition \ref{zoo}.
\end{remark}

Next, we define the weighted Musielak-Orlicz Sobolev space $\WV$  by
\begin{align*}
	\WV=\curly{u\in L^{\h}_V(\RN)\colon  |\nabla u|\in L^{\h}(\RN)},
\end{align*}
and it is equipped with the norm
\begin{align*}
	\|u\|_{\WV}=  \| |\nabla u| \|_{L^{\h}(\RN)}+\|u\|_{L^{\h}_V(\RN)}.
\end{align*}
Let us define $\rho\colon W_V^{1, \mathcal{H}}\left(\mathbb{R}^d\right) \rightarrow \mathbb{R}$ by $\rho(u)=\rho_{\mathcal{H}}(\nabla u)+\rho_{V, \mathcal{H}}(u)$, i.e.
\begin{align*}
	\rho(u)=\int_{\mathbb{R}^d} \h\L(x,|\nabla u|\right)\,\mathrm{d}x+\int_{\mathbb{R}^d} \h_V\L(x,| u|\right)\,\mathrm{d}x.
\end{align*}

The norm $\|\cdot\|_{\WV}$ and the modular $\rho$ have the following relationships. We refer to Zhang--Zuo--R\u{a}dulescu \cite[Lemma 2.4]{Zhang-Zuo-Radulescu-2024} for the case of constant exponent.

\begin{proposition}\label{RELHV}
	Let \eqref{H} and \eqref{V0} \textnormal{(i)} be satisfied. Then, the following hold:
	\begin{enumerate}
		\item[\textnormal{(i)}]
			If $\|u\|_{\WV}<1,$ then $\|u\|_{\WV}^{q^+}\leqslant \rho(u) \leqslant\|u\|^{p^-}_{\WV}$.
		\item[\textnormal{(ii)}]
			If $\|u\|_{\WV}>1,$ then $\|u\|^{p^-}_{\WV} \leqslant \rho(u) \leqslant\|u\|_{\WV}^{q^+}$.
		\item[\textnormal{(iii)}]
			$\|u\|_{\WV}\rightarrow 0 \Longleftrightarrow\rho_{}(u)\rightarrow 0$.
		\item[\textnormal{(iv)}]
			$\|u\|_{\WV}\rightarrow +\infty \Longleftrightarrow \rho(u)\rightarrow +\infty$.
	\end{enumerate}
\end{proposition}

\section{Embedding results}\label{Emb}

In this section, based on the papers by Bahrouni--Missaoui--Ounaies \cite{Bahrouni-Missaoui-Ounaies-2024} and Cianchi--Diening \cite{Cianchi-Diening-2024}, we present our main embedding results (continuous and compact) for the three spaces $\W$, $\WV$ and $\Wr$. Therefore, we divide this section into three parts.

\subsection{Continuous embedding of $\W$ (Theorem \ref{Inj1})}

We begin by stating the following lemma, which will be useful in the proof of Theorem \ref{Inj1}.

\begin{lemma}
	Let \eqref{H} \textnormal{(i)}, \textnormal{(ii)} be satisfied. Then, the embedding
	\begin{equation}\label{EQ6}
		\W \hookrightarrow V^{1,\h}_{\textrm{d}} (\RN).
	\end{equation}
	is continuous.
\end{lemma}

\begin{proof}
	First, it is clear that
	\begin{align*}
		\norm{u}_{V^{1,\h}_{\textrm{d}} (\RN)} \leq \norm{u}_{\W}.
	\end{align*}
	Hence, it suffices to prove that
	\begin{equation}\label{EQ2}
		\W \subseteq V^{1,\h}_{\textrm{d}} (\RN).
	\end{equation}
	To this end, let $u \in \W$. Arguing by contradiction, suppose that there exists $\l_0 > 0$ such that $\left\vert\curly{x \in \RN\colon \abs{u(x)} > \l_0}\right\vert = +\infty$. Since $u \in L^{\h}(\RN)$, one has
	\begin{equation}\label{EQ1}
		\int_{\RN}^{} \h(x,u) \,\mathrm{d}x < \infty.
	\end{equation}
	Moreover, using Proposition \ref{nfct}(ii), we obtain
	\begin{align*}
		\int_{\RN}^{} \h(x,u) \,\mathrm{d}x
		& \geq \int_{E:= \curly{x \in \RN\colon \abs{u} > \l_0}}^{} \h(x,u) \,\mathrm{d}x \geq \int_{E} \h(x,\l_0) \,\mathrm{d}x \geq \frac{1}{q^+} \int_{E} \l_0 h(x,\l_0) \,\mathrm{d}x\\
		&= \frac{1}{q^+} \int_{E} \left( \l_0^{p(x,\l_0)} +  \mu(x) \l_0^{q(x,u_0)} \right) \,\mathrm{d}x
		\geq \frac{1}{q^+} \min\curly{\l_0^{q^+}, 1}  \left\vert E\right\vert = +\infty,
	\end{align*}
	and this contradicts  \eqref{EQ1}. The inclusion \eqref{EQ2} is therefore established. The proof is complete.
\end{proof}

Now, we are in the position to prove Theorem \ref{Inj1}.

\begin{proof}[Proof of Theorem \ref{Inj1}]
	(i) First, invoking Proposition \ref{Prop1}, we know that $\h$ satisfies the conditions \textnormal{(A$_0$)}, \textnormal{(A$_1$)}, and \textnormal{(A$_2$)}. Moreover, since $p^+ < d$, we can see that
	\begin{align*}
		\int_{0}^{} \left(\frac{t}{\h_\infty (t)}\right)^{\frac{1}{d-1}} \, \mathrm{d}t  \leq \int_{0}^{}\frac{1}{t^{\frac{p^+-1}{d-1}}}\,\mathrm{d}t < +\infty.
	\end{align*}
	Therefore, according to Theorem 3.1 by Cianchi--Diening \cite{Cianchi-Diening-2024}, we have that
	\begin{equation}\label{EQ5}
		V^{1,\h}_{\textrm{d}} (\RN) \hookrightarrow L^{\h_\ast}(\RN).
	\end{equation}
	Combining \eqref{EQ6} and \eqref{EQ5}, we conclude that
	\begin{align*}
		\W \hookrightarrow V^{1,\h}_{\textrm{d}} (\RN) \hookrightarrow L^{\h_\ast}(\RN).
	\end{align*}
	Then, we conclude (i).

	(ii) Invoking conditions \eqref{2eq60} and \eqref{mla1b}, we can find $\delta, T > 0$ such that
	\begin{align*}
		\mathcal{V}(x,t) &\leq \h(x,t) \quad \text{for all } |t| \leq \delta \text{ and for all } x \in \mathbb{R}^d,\\
		\mathcal{V}(x,t) &\leq \h_\ast(x,t) \quad \text{for all } |t| \geq T \text{ and for all } x \in \mathbb{R}^d.
	\end{align*}
	From \eqref{cA}, \eqref{bf}, and Proposition \ref{zoo}, it follows that
	\begin{equation}\label{2eq778}
		\begin{aligned}
			& \int_{\mathbb{R}^d} \mathcal{V}(x,u(x)) \,\mathrm{d}x\\
			&\leq \int_{\{ |u| \geq T \}} \h_*(x,u(x)) \,\mathrm{d}x + \int_{\{ |u| \leq \delta \}} \h(x,u(x)) \,\mathrm{d}x + \int_{\{ \delta < |u| < T \}} \mathcal{V}(x,u(x)) \,\mathrm{d}x \\
			&\leq \int_{\{ |u| \geq T \}} \h_*(x,u(x)) \,\mathrm{d}x + \int_{\{ |u| \leq \delta \}} \h(x,u(x)) \,\mathrm{d}x + \int_{\{ \delta < |u| < T \}} \mathcal{V}(x,T) \,\mathrm{d}x \\
			&\leq \int_{\{ |u| \geq T \}} \h_*(x,u(x)) \,\mathrm{d}x + \int_{\{ |u| \leq \delta \}} \h(x,u(x)) \,\mathrm{d}x \\
			&\quad + \max\curly{ T^{v^-}, T^{v^+} } \int_{\{ \delta < |u| < T \}} \mathcal{V}(x,1) \,\mathrm{d}x \\
			&\leq \int_{\{ |u| \geq T \}} \h_*(x,u(x)) \,\mathrm{d}x + \int_{\{ |u| \leq \delta \}} \h(x,u(x)) \,\mathrm{d}x \\
			&\quad + C_{3} \max\{ T^{v^-}, T^{v^+} \} |\{ \delta < |u| < T \}|, \quad\text{for all }u \in \W
		\end{aligned}
	\end{equation}
	and for some constant $C_{3} > 0$. On the other hand, using \eqref{bf} and Proposition \ref{zoo}, we obtain
	\begin{equation}\label{ed290}
		\begin{aligned}
			|\{ \delta < |u| < T \}| &\leq \int_{\{ \delta < |u| < T \}} \frac{1}{\h(x,\delta)} \h(x,u(x)) \,\mathrm{d}x \\
			&\leq \frac{1}{\min\{ \delta^{p^-}, \delta^{q^+} \}} \int_{\{ \delta < |u| < T \}} \frac{1}{\h(x,1)} \h(x,u(x)) \,\mathrm{d}x \\
			&\leq \frac{1}{C_1 \min\{ \delta^{p^-}, \delta^{q^+} \}} \int_{\{ \delta < |u| < T \}} \h(x,u(x)) \,\mathrm{d}x \\
			&\leq \frac{1}{C_1 \min\{ \delta^{p^-}, \delta^{q^+} \}} \int_{\mathbb{R}^d} \h(x,u(x)) \,\mathrm{d}x.
		\end{aligned}
	\end{equation}
	Combining \eqref{2eq778} and \eqref{ed290}, we get our desired result. This concludes the proof.
\end{proof}

\subsection{Compact embedding of $\WV$}

In this subsection we give the proofs of Theorems \ref{Inj2}, \ref{Inj3} and \ref{Inj30}.

\begin{proof}[Proof of Theorem \ref{Inj2}]
	Let $\curly{ u_{n}}_{n \in \N} \subset \WV$ be a sequence such that $u_{n} \rightharpoonup u$ weakly in $\WV$. First, from \eqref{V0}(i), it is clear that $\WV \hookrightarrow \W$. Consequently, by applying \cite[Theorem 3.12]{Bahrouni-Bahrouni-Missaoui-Radulescu-2024}, we obtain that $u_{n} \rightarrow u$ strongly in $L^{\mathcal{H}}_{\operatorname{loc}}(\mathbb{R}^{d})$. We aim to demonstrate that $u_{n} \rightarrow u$ strongly in $L^{\h}(\RN)$. By Brézis-Lieb's theorem, see Proposition \ref{b-l}, it suffices to show that
	\begin{align*}
		\alpha_{n} := \int_{\RN} \h(x,u_{n})\,\mathrm{d}x \rightarrow \int_{\RN} \h(x,u)\,\mathrm{d}x.
	\end{align*}
	Given a bounded sequence $\curly{\alpha_{n}}_{n \in \N}$ in $\R$, we can extract a subsequence such that $\alpha_{n} \rightarrow \alpha$. Hence, by Fatou's lemma, we have
	\begin{align}\label{2eq47}
		\alpha \geq \int_{\RN} \h(x,u) \,\mathrm{d}x.
	\end{align}
	Utilizing \cite[Theorem 3.12]{Bahrouni-Bahrouni-Missaoui-Radulescu-2024}, we obtain
	\begin{align}\label{c}
		\int_{B_{r}} \h(x,u_{n})\,\mathrm{d}x \to \int_{B_{r}} \h(x,u)\,\mathrm{d}x,
	\end{align}
	where $B_{r} := B_{r}(0)$.\\
	\textbf{Claim:} For a given $\varepsilon > 0$, there exists $r = r(\varepsilon) > 0$ such that
	\begin{align*}
		\int_{B_{r}^{c}} \h(x,u_{n})\,\mathrm{d}x < \varepsilon,
	\end{align*}
	for $n \in \mathbb{N}$ sufficiently large.

	Indeed, for a given $\varepsilon > 0$, let $L > 0$ be such that
	\begin{align}\label{ml1}
		\frac{2}{\varepsilon} \max \left\{ M, \sup_n \max \left\lbrace \|u_n\|_{\WV}^{p^-}, \|u_n\|_{\WV}^{q^+} \right\rbrace \right\} < L,
	\end{align}
	where $M$ will be specified later. Applying Lemma \ref{Aux}, we can find a generalized \textnormal{N}-function $\mathcal{R}$ such that $\mathcal{R} \circ \h \ll \h_*$, which implies the existence of constants $C_4 > 0$ and $T > 0$ satisfying
	\begin{align}\label{2eq54}
		\mathcal{R}(x,\h(x,t)) \leq C_4 \h_*(x,t) \quad \text{for all } x \in \mathbb{R}^d \text{ and for all } t \geq T.
	\end{align}
	Now, we define the function $f\colon [0,+\infty) \to \mathbb{R}$ by
	\begin{align*}
		f(t) := \max \left\lbrace t^{\frac{r^+-1}{r^+ }}, t^{\frac{r^-- 1}{r^- }} \right\rbrace \quad \text{for all } t \geq 0,
	\end{align*}
	where $r^-$ and $r^+$ are defined in Lemma \ref{Aux}. It is clear that $f$ is continuous and $\ds\lim_{t \to 0} f(t) = 0$. According to \eqref{V0}(ii), we can choose $r > 1$ sufficiently large such that
	\begin{equation}\label{ml2}
		f \left( \left| \curly{ x \in B_r^c(0) \colon  V(x) < L } \right| \right) \leq \frac{\varepsilon}{2M}.
	\end{equation}

	Next, define the sets
	\begin{align*}
		\mathcal{C} := \curly{ x \in B_r^c(0) \colon  V(x) \geq L }
		\quad \text{and} \quad
		\mathcal{B} := \curly{ x \in B_r^c(0) \colon  V(x) < L }.
	\end{align*}
	By Proposition \ref{zoo} and \eqref{ml1}, one has
	\begin{align}\label{num1}
		\int_{\mathcal{C}} \h(x,u_n) \,\mathrm{d}x \leq  \int_{\mathcal{C}} \frac{V(x)}{L} \h(x,u_n) \,\mathrm{d}x
		\leq \frac{1}{L} \max \left\lbrace  \|u_n\|_{\WV}^{p^-}, \|u_n\|_{\WV}^{q^+} \right\rbrace
		\leq \frac{\varepsilon}{2}.
	\end{align}
	On the other hand, from Lemma \ref{Aux} and H\"{o}lder's inequality, we obtain
	\begin{align}\label{ineq1}
	  \int_{\mathcal{B}} \h(x,u_n) \,\mathrm{d}x \leq 2 \left\| \h(x,u_n) \right\|_{L^{\mathcal{R}}(\mathcal{B})} \left\| \chi_{\mathcal{B}} \right\|_{L^{\widetilde{\mathcal{R}}}(\mathbb{R}^d)},
	\end{align}
	where $\widetilde{\mathcal{R}}$ is the conjugate function of $\mathcal{R}$. Using Lemmas \ref{lm1}, \ref{Aux1}, and \ref{Aux}, we find
	\begin{align}\label{ineq2}
		\| \chi_{\mathcal{B}} \|_{L^{\widetilde{\mathcal{R}}}(\mathbb{R}^d)} \leq C_5 \max \left\lbrace |\mathcal{B}|^{\frac{r^+-1}{r^+}}, |\mathcal{B}|^{\frac{r^--1}{r^-}} \right\rbrace =C_5 f(|\mathcal{B}|).
	\end{align}
	Applying \eqref{2eq54}, Proposition \ref{zoo}, and Lemma \ref{Aux}(2), we deduce that
	\begin{equation}\label{ineq3}
		\begin{aligned}
			&\int_{\mathcal{B}} \mathcal{R}(x,\h(x,u_n)) \,\mathrm{d}x\\
			&= \int_{\mathcal{B} \cap [|u_n| \leq T]} \mathcal{R}(x,\h(x,u_n)) \,\mathrm{d}x + \int_{\mathcal{B} \cap [|u_n| > T]} \mathcal{R}(x,\h(x,u_n)) \,\mathrm{d}x \\
			&\leq \int_{\mathcal{B} \cap [|u_n| \leq T]} \mathcal{R}(x,\h(x,T)) \,\mathrm{d}x + C_4 \sup_n \int_{\mathbb{R}^d} \h_*(x,u_n) \,\mathrm{d}x \\
			&\leq \int_{\mathcal{B} \cap [|u_n|  \leq T]} \mathcal{R} \left(x, \max \left\lbrace T^{p^-}, T^{q^+} \right\rbrace \h(x,1) \right) \,\mathrm{d}x + C_4 \sup_n \int_{\mathbb{R}^d} \h_*(x,u_n) \,\mathrm{d}x \\
			&\leq \int_{\mathcal{B} \cap [|u_n| \leq T]} \mathcal{R} \left(x, \max \curly{ T^{p^-}, T^{q^+}} C_2  \right)  \,\mathrm{d}x + C_4\sup_n \int_{\mathbb{R}^d} \h_*(x,u_n) \,\mathrm{d}x \\
			&\leq \int_{\mathcal{B} \cap [|u_n| \leq T]}  \max \curly{ \left( \max \curly{ T^{p^-}, T^{q^+} } C_2  \right)^{r^-}, \left( \max \curly{ T^{p^-}, T^{q^+} } C_2 \right)^{r^+} } \mathcal{R}(x,1) \,\mathrm{d}x \\
			&\quad + C_4 \sup_n \int_{\mathbb{R}^d} \h_*(x,u_n) \,\mathrm{d}x \\
			&\leq \int_{\mathcal{B} \cap [|u_n| \leq T]}  c_2 \max \curly{ \left( \max \curly{ T^{p^-}, T^{q^+} } C_2 \right)^{r^-}, \left( \max \curly{ T^{p^-}, T^{q^+} } C_2 \right)^{r^+} } \,\mathrm{d}x \\
			&\quad + C_4 \sup_n \int_{\mathbb{R}^d}  \h_*(x,u_n) \,\mathrm{d}x \\
			&\leq |\mathcal{B}_1| T_1 + C_4 \sup_n  \int_{\mathbb{R}^d} \h_*(x,u_n) \,\mathrm{d}x,
		\end{aligned}
	\end{equation}
	where $\mathcal{B} \subset \mathcal{B}_1$, for all $0 < \varepsilon < 1$ small enough, and
	\begin{align*}
		T_1 := c_2 \max \curly{ \left( \max \curly{ T^{p^-}, T^{q^+} }C_2 \right)^{r^-}, \left( \max \curly{ T^{p^-}, T^{q^+} } C_2 \right)^{r^+}}.
	\end{align*}
	Next, we define
	\begin{align*}
		M := 2C_5 |\mathcal{B}_1| T_1 + 2C_4 C_5 \sup_n \int_{\mathbb{R}^d} \h_*(x,u_n) \,\mathrm{d}x + 2C_5.
	\end{align*}
	Thus, by \eqref{ml2}, \eqref{ineq1}, \eqref{ineq2}, and \eqref{ineq3}, we obtain
	\begin{align}\label{num2}
		\int_{\mathcal{B}} \h(x,u_n) \,\mathrm{d}x \leq B_2 \max \curly{ |\mathcal{B}|^{\frac{r^+-1}{r^+}}, |\mathcal{B}|^{\frac{r^--1}{r^-}} } = M f(|\mathcal{B}|) \leq \frac{\varepsilon}{2}.
	\end{align}
	Therefore, in light of \eqref{num1} and \eqref{num2}, we conclude that
	\begin{align*}
		\int_{B_r^c(0)} \h(x,u_n) \,\mathrm{d}x = \int_{\mathcal{C}} \h(x,u_n) \,\mathrm{d}x + \int_{\mathcal{B}} \h(x,u_n) \,\mathrm{d}x < \varepsilon.
	\end{align*}
	This shows the Claim.

	Exploiting the Claim and \eqref{c}, we find that
	\begin{equation}\label{2eq48}
		\begin{aligned}
			\int_{\mathbb{R}^d} \h(x,u) \,\mathrm{d}x
			&= \int_{B_r(0)} \h(x,u) \,\mathrm{d}x + \int_{B_r^c(0)} \h(x,u) \,\mathrm{d}x \\
			&\geq \lim_{n \to \infty} \int_{B_r(0)} \h(x,u_n) \,\mathrm{d}x \\
			&= \lim_{n \to \infty} \int_{\mathbb{R}^d} \h(x,u_n)\,\mathrm{d}x - \lim_{n \to \infty} \int_{B_r^c(0)} \h(x,u_n) \,\mathrm{d}x \\
			&\geq \alpha - \varepsilon.
		\end{aligned}
	\end{equation}
	Consequently, combining \eqref{2eq47} and \eqref{2eq48}, we deduce the proof of the theorem.
\end{proof}

\begin{proof}[Proof of  Theorem \ref{Inj3}]
	Given $\mathcal{V} \ll \h^{*}$, for any $\varepsilon > 0$, we can find $T > 0$ such that
	\begin{align}\label{ml4}
		\frac{\mathcal{V}(x, |t|)}{\h^{*}(x, |t|)} \leq \frac{\varepsilon}{2\kappa} \quad \text{for } |t| \geq T  \text{ and for all }  x \in \mathbb{R}^d,
	\end{align}
	where $\kappa > 0$ will be specified later. Consider the sequence $\{u_{n}\}_{n \in \mathbb{N}} \subset \WV$ such that $u_{n} \rightharpoonup 0$ in $\WV$. According to Theorem \ref{Inj2}, this implies that
	\begin{align}\label{2eq20}
		u_{n} \rightarrow 0 \quad \text{in }  L^{\h}(\mathbb{R}^{d}).
	\end{align}
	We then decompose the integral as follows
	\begin{align}\label{ml5}
		\int_{\mathbb{R}^{d}} \mathcal{V}(x, |u_{n}|) \,\mathrm{d}x = \int_{\{|u_{n}| \geq T\}} \mathcal{V}(x, |u_{n}|) \,\mathrm{d}x + \int_{\{|u_{n}| < T\}} \mathcal{V}(x, |u_{n}|) \,\mathrm{d}x.
	\end{align}
	By invoking Theorem \ref{Inj1}, we set
	\begin{align*}
		\kappa := \sup_{n} \int_{\mathbb{R}^{d}} \h^{*}(x, |u_{n}|) \,\mathrm{d}x < +\infty.
	\end{align*}
	From \eqref{ml4}, one has
	\begin{align}\label{ml6}
		\int_{\{|u_{n}| \geq T\}} \mathcal{V}(x, |u_{n}|) \,\mathrm{d}x \leq \frac{\varepsilon}{2\kappa} \int_{\mathbb{R}^{d}} \h^{*}(x, |u_{n}|) \,\mathrm{d}x \leq \frac{\varepsilon}{2}.
	\end{align}
	To complete the proof, we need to analyze the integral in \eqref{ml5} over the set $\{|u_{n}| < T\}$. For this purpose, we will use either assumption \eqref{mla1} or \eqref{mla2}.\\
	\textbf{Under assumption \eqref{mla1}:}

	Let $\theta \in (0, 1)$. By applying H\"{o}lder's inequality, we obtain
	\begin{align}\label{ml9}
		\int_{\{ |u_n| < T \}} \mathcal{V}(x,|u_n|) \,\mathrm{d}x \leq \left[ \int_{\{ |u_n| < T \}} \left( \frac{\mathcal{V}(x,|u_n|)}{\h(x,|u_n|)^{\theta}} \right)^{\frac{1}{1-\theta}} \,\mathrm{d}x \right]^{1-\theta} \left[ \int_{\mathbb{R}^d} \h(x,|u_n|) \,\mathrm{d}x \right]^\theta.
	\end{align}
	By assumption \eqref{mla1}, there exist constants $\delta, C_6 > 0$ such that
	\begin{align*}
		\mathcal{V}(x,|u_n|) \leq C_6 \h(x,|u_n|) \quad \text{for all } |u_n| \leq \delta \text{ and for all } x \in \mathbb{R}^d.
	\end{align*}
	If $\delta < T$, then using \eqref{bf} and Proposition \ref{zoo}, we have
	\begin{align*}
		\frac{\mathcal{V}(x,|u_n|)}{\h(x,|u_n|)} \leq \frac{\mathcal{V}(x,T)}{\h(x,\delta)} \leq \frac{\max \left\{ T^{v^+}, T^{v^-} \right\} \mathcal{V}(x,1)}{\min \curly{ \delta^{p^-}, \delta^{q^+} } \h(x,1)} \leq \widetilde{C} \frac{\max \curly{ T^{v^-}, T^{v^+} }}{\min \curly{ \delta^{p^-}, \delta^{q^+} }},
	\end{align*}
	for $|u_n| \in [\delta, T]$ and $x \in \mathbb{R}^d$ with some constant $\widetilde{C} > 0$ independent of $x$. Hence,
	\begin{align*}
		\frac{\mathcal{V}(x,|u_n|)}{\h(x,|u_n|)} \leq \widetilde{C}_1^{1-\theta} \quad \text{for all } |u_n| \leq T \text{ and for all } x \in \mathbb{R}^d,
	\end{align*}
	where $\widetilde{C}_1^{1-\theta} := \max \left\{ C_6, \widetilde{C} \dfrac{\max \{ T^{v^-}, T^{v^+} \}}{\min \{ \delta^{p^-}, \delta^{q^+} \}} \right\}$. Consequently,
	\begin{align}\label{ml7}
		\left( \frac{\mathcal{V}(x,|u_n|)}{\h(x,|u_n|)^\theta} \right)^{\frac{1}{1-\theta}} \leq \widetilde{C}_1 \h(x,|u_n|) \quad \text{for all } |u_n| \leq T \text{ and for all } x \in \mathbb{R}^d.
	\end{align}
	From \eqref{2eq20}, there exists $n_0 \in \mathbb{N}$ such that
	\begin{align}\label{ml8}
		\int_{\mathbb{R}^d} \h(x,|u_n|) \,\mathrm{d}x < \frac{\varepsilon}{2 \widetilde{C}_1^{1-\theta}} \quad \text{for all } n > n_0.
	\end{align}
	Therefore, combining \eqref{ml6}, \eqref{ml9}, \eqref{ml7}, and \eqref{ml8}, we deduce that
	\begin{align}\label{mla8}
		\int_{\mathbb{R}^d} \mathcal{V}(x,|u_n|) \,\mathrm{d}x < \varepsilon.
	\end{align}
	This completes the proof.\\
	\textbf{Under assumption \eqref{mla2}:}

	Fix $\varepsilon >0$. Then, from \eqref{2eq20}, there exists $n_0 \in \mathbb{N}$ such that
	\begin{align}\label{2eq55}
		\int_{\mathbb{R}^d} \h(x,|u_n|) \,\mathrm{d}x \leq \min \left\{ \frac{\varepsilon}{4\widetilde{C}_3}, \left(\frac{\varepsilon}{4k^{1-a}}\right)^{\frac{1}{a}} \right\} \quad \text{for all } n \geq n_0,
	\end{align}
	where $\widetilde{C}_3$ will be defined later.

	Again, by \eqref{2eq20} and assumption \eqref{mla2}, we obtain
	\begin{equation}\label{2eq56}
		\begin{aligned}
			\int_{\{|u_n| \leq 1\}} \mathcal{V}(x,|u_n|) \,\mathrm{d}x
			&\leq \int_{\{|u_n| \leq 1\}} \h(x,|u_n|)^a  \h_*(x,|u_n|)^{1-a} \,\mathrm{d}x \\
			&\leq \left( \int_{\{|u_n| \leq 1\}} \h(x,|u_n|)  \,\mathrm{d}x \right)^a \left( \int_{\{|u_n| \leq 1\}} \h_*(x,|u_n|) \,\mathrm{d}x \right)^{1-a} \\
			&\leq \frac{\varepsilon}{4}.
		\end{aligned}
	\end{equation}

	If $T > 1$, then by \eqref{2eq55}, \eqref{bf}, and Proposition \ref{zoo}, for all $n \geq n_0$, we have
	\begin{equation}\label{2eq57}
		\begin{aligned}
			\int_{\{1 \leq |u_n| \leq T\}} \mathcal{V}(x,|u_n|) \,\mathrm{d}x
			&\leq \int_{\mathbb{R}^d} \frac{\mathcal{V}(x,T)}{\h(x,1)} \h(x,|u_n|) \,\mathrm{d}x \\
			&\leq \max \left\{ T^{v^+}, T^{v^-} \right\} \int_{\mathbb{R}^d} \frac{\mathcal{V}(x,1)}{\h(x,1)} \h(x,|u_n|) \,\mathrm{d}x  \\
			&\leq \widetilde{C}_2 \max \left\{ T^{v^+}, T^{v^-} \right\} \int_{\mathbb{R}^d} \h(x,|u_n|) \,\mathrm{d}x \\
			&\leq \widetilde{C}_3 \int_{\mathbb{R}^d}  \h(x,|u_n|) \,\mathrm{d}x\\
			&< \frac{\varepsilon}{4},
		\end{aligned}
	\end{equation}
	for some constants $\widetilde{C}_2 > 0$ and $\widetilde{C}_3 := \widetilde{C}_2 \max \left\{ T^{v^+}, T^{v^-} \right\}$. Thus, by combining \eqref{ml6}, \eqref{2eq55}, \eqref{2eq56}, and \eqref{2eq57}, we conclude that \eqref{mla8} holds.
\end{proof}

\begin{proof}[Proof of  Theorem \ref{Inj30}]
	Let us consider the weighted Sobolev space
	\begin{align*}
		W^{1,p^-}_V(\RN):=\curly{u \in W^{1,p^-}(\RN)\colon \int_{\RN} V(x) |u|^{p^-}\,\mathrm{d}x <+\infty},
	\end{align*}
	endowed with the norm
	\begin{align*}
		\|u\|_{W^{1,p^-}_V(\RN)} := \|u\|_{L^{p^-}_V(\RN)}+\|\nabla u\|_{L^{p^-}(\RN)}.
	\end{align*}
	The proof of Theorem \ref{Inj30} can be directly obtained by noting that $\WV \hookrightarrow W^{1,p^-}_V(\RN)$ and invoking Lemma 7 by Stegli\'{n}ski \cite{Steglinski-2022}.
\end{proof}

\subsection{ Compact embedding of $\Wr$ (Theorem \ref{thms})}

In this paragraph, we focus on proving the Strauss-type embedding theorem given in Theorem \ref{thms}. For this purpose, we first need to demonstrate the Lions-type lemma given in Theorem \ref{lions}.

\begin{proof}[Proof of Theorem \ref{lions}]
	Let $\{ u_{n} \}_{n \in \mathbb{N}} \subset \W$ be a sequence satisfying \eqref{lionss}. Given $\mathcal{V} \ll \h_{*}$, for any $\varepsilon > 0$, there exists $T > 0$ such that
	\begin{align}\label{l1}
		\frac{\mathcal{V}(x,|t|)}{\h_{*}(x,|t|)} \leq \frac{\varepsilon}{3\kappa} \quad \text{for all } |t| \geq T \text{ and for all } x \in \mathbb{R}^d,
	\end{align}
	where $\kappa$ is defined as
	\begin{align*}
		\kappa = \sup_n \int_{\mathbb{R}^d} \h_{*}(x,|u_n|) \,\mathrm{d}x < +\infty.
	\end{align*}
	From \eqref{mla1b}, there exists $\delta > 0$ such that
	\begin{align}\label{l3}
		\frac{\mathcal{V}(x,|t|)}{\h(x,|t|)} \leq \frac{\varepsilon}{3\theta} \quad \text{for all } |t| < \delta \text{ and for all } x \in \mathbb{R}^d,
	\end{align}
	where
	\begin{align*}
		\theta = \sup_n \int_{\mathbb{R}^d} \h(x,|u_n|) \,\mathrm{d}x < +\infty.
	\end{align*}
	Consider the decomposition
	\begin{align*}
		\int_{\mathbb{R}^d} \mathcal{V}(x,|u_n|) \,\mathrm{d}x = \int_{\{|u_n| \leq \delta\}} \mathcal{V}(x,|u_n|) \,\mathrm{d}x + \int_{\{\delta < |u_n| < T\}} \mathcal{V}(x,|u_n|) \,\mathrm{d}x + \int_{\{|u_n| \geq T\}} \mathcal{V}(x,|u_n|) \,\mathrm{d}x.
	\end{align*}
	Using \eqref{l1}, we obtain
	\begin{align}\label{mlb0}
		\int_{\{|u_n| \geq T\}} \mathcal{V}(x,|u_n|) \,\mathrm{d}x \leq \frac{\varepsilon}{3\kappa} \int_{\mathbb{R}^d} \h_{*}(x,|u_n|) \,\mathrm{d}x \leq \frac{\varepsilon}{3}.
	\end{align}
	Similarly, from \eqref{l3}, we have
	\begin{align}\label{mlb1}
		\int_{\{|u_n| \leq \delta\}} \mathcal{V}(x,|u_n|) \,\mathrm{d}x \leq \frac{\varepsilon}{3\theta} \int_{\{|u_n| \leq \delta\}} \h(x,|u_n|) \,\mathrm{d}x \leq \frac{\varepsilon}{3}.
	\end{align}

	In the rest of the proof, we consider two cases\\
	\underline{\bf Case 1:} Suppose
	\begin{align*}
		\lim_{n \to \infty} |\{\delta < |u_n| < T\}| = 0.
	\end{align*}
	Thus, there exists $n_0 \in \mathbb{N}$ such that
	\begin{align}\label{ml10}
		|\{\delta < |u_n| < T\}| < \frac{\varepsilon}{3\widetilde{C}_5\widetilde{C}_4} \quad \text{for all }n \geq n_0,
	\end{align}
	where $\widetilde{C}_4, \widetilde{C}_5 > 0$ are constants defined later. Consequently, we have
	\begin{equation}\label{ed2}
		\begin{aligned}
			\int_{\{\delta < |u_n| < T\}} \h(x,u_n)\,\mathrm{d}x
			&\leq  \int_{\{\delta < |u_n| < T\}} \h(x,T) \,\mathrm{d}x \\
			&\leq  \max\curly{T^{p^-}, T^{q^+}} \int_{\{\delta < |u_n| < T\}} \h(x,1) \,\mathrm{d}x\,  \\
			&\leq C_2 \max\curly{T^{p^-}, T^{q^+}} |\{\delta < |u_n| < T\}|\,  \\
			&=  \widetilde{C}_5 |\{\delta < |u_n| < T\}|\,  \leq \frac{\varepsilon}{3\widetilde{C}_4}.
		\end{aligned}
	\end{equation}
	For $n \geq n_0$, using Proposition \ref{zoo}, \eqref{bf}, \eqref{ml10}, and \eqref{ed2} we find
	\begin{equation}\label{ml12}
		\begin{aligned}
			\int_{\{\delta < |u_n| < T\}} \mathcal{V}(x,u_n) \,\mathrm{d}x
			&\leq \int_{\{\delta < |u_n| < T\}} \frac{\mathcal{V}(x,T)}{\h(x,\delta)} \h(x,u_n) \,\mathrm{d}x  \\
			&\leq \frac{\max\{T^{v^+}, T^{v^-}\}}{\min\{\delta^{p^-}, \delta^{q^+}\}} \int_{\{\delta < |u_n| < T\}} \frac{\mathcal{V}(x,1)}{\h(x,1)} \h(x,u_n) \,\mathrm{d}x \\
			&\leq \widetilde{C}_3 \frac{\max\{T^{v^+},  T^{v^-}\}}{\min\{\delta^{p^-}, \delta^{q^+}\}} \int_{\{\delta < |u_n| < T\}} \h(x,u_n) \,\mathrm{d}x \\
			&\leq \widetilde{C}_4 \int_{\{\delta < |u_n| < T\}} \h(x,u_n) \,\mathrm{d}x  \\
			&< \frac{\varepsilon}{3},
		\end{aligned}
	\end{equation}
	for some $\widetilde{C}_3 > 0$ and $\widetilde{C}_4 = \widetilde{C}_3 \dfrac{\max\{T^{v^+}, T^{v^-}\}}{\min\{\delta^{p^-}, \delta^{q^+}\}}$. Therefore, using \eqref{mlb0}, \eqref{mlb1}, and \eqref{ml12}, we conclude that
	\begin{align*}
		\int_{\mathbb{R}^d} \mathcal{V}(x,u_n) \,\mathrm{d}x \leq \varepsilon \quad \text{for any } \varepsilon > 0.
	\end{align*}
	This completes the proof for the first case.\\
	\underline{\bf Case 2:} Assume, up to a subsequence, that
	\begin{align*}
		\lim_{n \to \infty} |\{\delta < |u_n| < T\}| = M \in (0, \infty).
	\end{align*}
	We will show that this case cannot occur. Consider the following claim:\\
	\textbf{Claim:} There exist $ y_0 \in \mathbb{R}^d $ and $ \sigma > 0 $ such that
	\begin{align*}
		0 < \sigma \leq |\{\delta < |u_n| < T\} \cap B_r(y_0)|
	\end{align*}
	holds true for a subsequence of $ \{ u_n \}_{n \in \mathbb{N}} $, which is also labeled as $ \curly{ u_n}_{n \in \mathbb{N}} $. The proof follows by contradiction. Indeed, for each $ \varepsilon > 0 $ and $ k \in \mathbb{N} $, we have
	\begin{align}\label{ed7}
		|\{\delta < |u_n| < T\} \cap B_r(y)| < \frac{\varepsilon}{2^k}
	\end{align}
	for all $ y \in \mathbb{R}^d $. Notice that the last estimate holds for any subsequence of $ \curly{ u_n }_{n \in \mathbb{N}}$. Without loss of generality, we consider just the sequence $ \curly{ u_n }_{n \in \mathbb{N}}$. Now, choose $ \{ y_k \}_{k \in \mathbb{N}} \subset \mathbb{R}^d $ such that $ \ds \cup_{k=1}^{\infty} B_r(y_k) = \mathbb{R}^d $. Using \eqref{ed7}, we write
	\begin{align*}
		|\{\delta < |u_n| < T\}|
		&= |\{\delta < |u_n| < T\} \cap (\cup_{k=1}^{\infty} B_r(y_k))| \\
		&\leq \sum_{k=1}^{\infty} |\{\delta < |u_n| < T\} \cap B_r(y_k)|  \\
		&\leq \sum_{k=1}^{\infty} \frac{\varepsilon}{2^k} = \varepsilon,
	\end{align*}
	where $ \varepsilon > 0 $ is arbitrary. Up to a subsequence, it follows from the last estimate that
	\begin{align*}
		0 < M = \lim_{n \to \infty} |\{\delta < |u_n| < T\}| \leq \varepsilon,
	\end{align*}
	which does not make sense for $ \varepsilon \in (0, M) $. Thus, the proof of the claim is finished.

	At this point, using the Claim, \eqref{lionss}, \eqref{bf}, and Proposition \ref{zoo}, we observe that
	\begin{align*}
		0 &< \sigma \leq |\{\delta < |u_n| < T\}  \cap B_r(y_0)| \leq \int_{B_r(y_0)} \frac{1}{\h(x,\delta)}  \h(x,u_n) \,\mathrm{d}x \nonumber \\
		&\leq \min \left\{ \delta^{p^-}, \delta^{q^+} \right\} \int_{B_r(y_0)} \frac{1}{\h(x,1)} \h(x,u_n) \,\mathrm{d}x  \\
		&\leq \widetilde{C}_6 \min \left\{  \delta^{p^-},  \delta^{q^+} \right\} \sup_{y \in \mathbb{R}^d} \int_{B_r(y)} \h(x,u_n) \,\mathrm{d}x \to 0 \quad \text{as } n \to \infty.
	\end{align*}
	This contradiction proves that the second case is impossible. In other words, we prove that $ M = 0 $ is always verified. Hence, our result follows from the first case. This ends the proof.
\end{proof}

\begin{proof}[Proof of the Strauss embedding theorem (Theorem \ref{thms})]
	Let $ \{ u_n \}_{n \in \mathbb{N}} \subset W^{1,\h}_{\operatorname{rad}}(\mathbb{R}^d) $ be a bounded sequence. Since $ W^{1,\h}_{\operatorname{rad}}(\mathbb{R}^d) $ is a reflexive space, up to a subsequence, still denoted by $ \curly{ u_n }_{n \in \mathbb{N}}$,
	\begin{align*}
		u_n \rightharpoonup u \quad\text{in } W^{1,\h}_{\operatorname{rad}}(\mathbb{R}^d).
	\end{align*}
	Without loss of generality, we may assume that $ u \equiv 0 $. Using the continuous embedding $ W^{1,\h}(\mathbb{R}^d) \hookrightarrow L^{\h}(\mathbb{R}^d) $, we can find a constant $ C_7 > 0 $ such that
	\begin{align}\label{2eq63}
		\int_{\mathbb{R}^d} \h(x,u_n) \,\mathrm{d}x < C_7.
	\end{align}
	Let us fix $ r > 0 $. Since $ u_n $ is radially symmetric for all $ n \in \mathbb{N} $,
	\begin{align}\label{2eq64}
		\int_{B_r(y_1)} \h(x,u_n) \,\mathrm{d}x = \int_{B_r(y_2)} \h(x,u_n) \,\mathrm{d}x\quad \text{for all } y_1, y_2 \in \mathbb{R}^d  \text{ with }  |y_1| = |y_2|.
	\end{align}
	In the sequel, for each $ y \in \mathbb{R}^d $, $ |y| > r $, we denote by $ \gamma(y) $ the maximum number of integers $ j \geq 1 $ such that there exist $ y_1, y_2, \ldots, y_j \in \mathbb{R}^d $,
	with
	\begin{align*}
		|y_1| = |y_2| = \cdots = |y_j| = |y|
		\quad\text{and}\quad
		B_r(y_i) \cap B_r(y_k) = \emptyset \quad\text{whenever } i \neq k.
	\end{align*}
	From the above definition, it is clear that
	\begin{align}\label{2eq65}
		\gamma(y) \to +\infty \quad \text{as }  |y| \to +\infty.
	\end{align}
	Let $ y \in \mathbb{R}^d $, $ |y| > r $, and choose $ y_1, \ldots, y_{\gamma(y)} \in \mathbb{R}^d $ as above. Thus, by \eqref{2eq63}, \eqref{2eq64}, and \eqref{2eq65}, we obtain
	\begin{align*}
		C_7
		&> \int_{\mathbb{R}^d} \h(x,u_n) \,\mathrm{d}x \geq 	\sum_{i=1}^{\gamma(y)} \int_{B_r(y_i)} \h(x,u_n) \,\mathrm{d}x \\
		&\geq \gamma(y) \int_{B_r(y)} \h(x,u_n) \,\mathrm{d}x.
	\end{align*}
	It follows, by \eqref{2eq65}, that
	\begin{align*}
		\int_{B_r(y)} \h(x,u_n) \,\mathrm{d}x \leq \frac{C_7}{\gamma(y)} \to 0 \quad \text{as } |y| \to +\infty.
	\end{align*}
	Therefore, for arbitrary $ \varepsilon > 0 $, there exists $ R_\varepsilon > 0 $ such that
	\begin{align}\label{2eq67}
		\sup_{|y| \geq R_\varepsilon} \int_{B_r(y)} \h(x,u_n) \,\mathrm{d}x \leq \varepsilon \quad\text{for } n \in \mathbb{N}.
	\end{align}
	On the other hand, by \cite[Theorem 3.12]{Bahrouni-Bahrouni-Missaoui-Radulescu-2024}, we have the compact embedding
	\begin{align*}
		W^{1,\h}(B_{r+R_\varepsilon}(0)) \hookrightarrow L^{\h}(B_{r+R_\varepsilon}(0)).
	\end{align*}
	Hence, $ u_n \to 0 $ in $ L^{\h}(B_{r+R_\varepsilon}(0)) $, and
	\begin{align*}
		\int_{B_{r+R_\varepsilon}(0)} \h(x,u_n) \,\mathrm{d}x \to 0 \quad \text{as }  n \to +\infty,
	\end{align*}
	which implies that
	\begin{align}\label{2eq69}
		\sup_{|y| < R_\varepsilon} \int_{B_r(y)} \h(x,u_n) \,\mathrm{d}x \to 0 \quad \text{as } n \to +\infty.
	\end{align}
	Putting \eqref{2eq67} and \eqref{2eq69} together, and applying Theorem \ref{lions}, we deduce that
	\begin{align*}
		u_n \to 0\quad \text{in } L^{\mathcal{V}}(\mathbb{R}^d).
	\end{align*}
\end{proof}

\section{Application}\label{App}

In this section, we prove our main  multiplicity result. Recall that the problem under consideration is the following one
\begin{equation*}
	\begin{aligned}
		-\operatorname{div}&\left(|\nabla u|^{p(x,|\nabla u|)-2} \nabla u+ \mu(x)| \nabla u|^{q(x,|\nabla u|)-2}  \nabla u\right)\\&+ V(x)\left(|u|^{p(x,| u|)-2}  u+ \mu(x)|  u|^{q(x,| u|)-2}  u\right)=\l f(x,u), \ \ x \in \mathbb{R}^d,
	\end{aligned}
\end{equation*}

First, we recall the definition of the Cerami condition, which will be needed.
\begin{definition} \label{dcrmi}
	Let $X$ be a Banach space, and denote by $X^\ast$ its topological dual space. Given $L \in C^1(X)$, we say that $L$ satisfies the Cerami-condition, \textnormal{(C)}-condition for short, if every sequence $\left\{u_n\right\}_{n \in \mathbb{N}} \subseteq X$ such that
	\begin{enumerate}
		\item[\textnormal{(C$_1$)}]
			$ \left\{L\left(u_n\right)\right\}_{n \geq 1} \subseteq \mathbb{R}$ is bounded,
		\item[\textnormal{(C$_2$)}]
			$ \left(1+\left\|u_n\right\|_X\right) L^{\prime}\left(u_n\right) \rightarrow 0$ in $X^*$ as $n \rightarrow \infty$,
		\end{enumerate}
		admits a strongly convergent subsequence in $X$.
\end{definition}

To prove Theorem \ref{thex}, we will use the following theorem, which can be found in the paper by Bonanno and D’Aguì \cite[Theorem 2.1 and Remark 2.2]{Bonanno-DAgui-2016}.

\begin{theorem}\label{bonano}
	Let $X$ be a real Banach space and let $\rho, K\colon  X \rightarrow \mathbb{R}$ be two continuously G\^ateaux differentiable functionals such that
	\begin{align*}
		\inf _X \rho=\rho(0)=K(0)=0.
	\end{align*}
	Assume that $\rho$ is coercive and there exist $r \in \mathbb{R}$ and $\tilde{u} \in X$, with $0<\rho(\tilde{u})<r$, such that
	\begin{equation}\label{32}
		\frac{\ds\sup _{u \in \rho^{-1}([-\infty, r])} K(u)}{r}<\frac{K(\tilde{u})}{\rho(\tilde{u})}
	\end{equation}
	and, for each $\lambda \in\left[\dfrac{\rho(\tilde{u})}{K(\tilde{u})}, \dfrac{r}{\ds\sup _{u \in \rho^{-1}([-\infty, r])} K(u)}\right],$ the functional $J_\lambda =\rho-\l K$ satisfies the \textnormal{(C)}-condition and it is unbounded
	from below. Then, for each  $\lambda \in\left[\dfrac{\rho(\tilde{u})}{K(\tilde{u})}, \dfrac{r}{\ds\sup _{u \in \rho^{-1}([-\infty, r])} K(u)}\right]$, the functional  $J_\lambda$ admits at least two nontrivial critical points
	$u_{\lambda, 1}, u_{\lambda, 2}$ such that $J_\lambda\left(u_{\lambda, 1}\right)<0<J_\lambda\left(u_{\lambda, 2}\right)$.
\end{theorem}

As shown in the statement of Theorem \ref{bonano}, the proof of our existence result depends on certain properties of the energy functional associated with the problem \eqref{prb}. Therefore, we divide this section into two parts. First, we examine the properties of the corresponding energy functional. In the second part, we check the geometric conditions of Theorem \ref{bonano} to ensure the multiplicity of weak solutions.

\subsection{Some properties of the energy functional} \label{Prty}

In this subsection, we work under the assumptions of Theorem \ref{thex}. We introduce the new double phase operator defined in \eqref{oper}, which is associated with our problem \eqref{prb}, along with its corresponding energy functional. Let $\L( \WV \r)^\ast $ be the  dual space of $\WV$ with its duality pairing denoted by $\langle\cdot, \cdot\rangle$. We say that $u \in \WV$ is a weak solution of problem \eqref{prb}, if
\begin{align*}
	&\int_{\RN}^{} \L(|\nabla u|^{p\L(x, |\nabla u| \r)-2}\nabla u\cdot \nabla v +\mu (x) \L(|\nabla u|^{q\L(x, |\nabla u| \r)-2}\nabla u\cdot\nabla v \r)\r)\,\mathrm{d}x\\
	&+ \int_{\RN}^{}V(x) \L(| u|^{p\L(x, |u| \r)-2} u  v +\mu (x) \L( |u|^{q\L(x, | u| \r)-2} u  v \r)\r)\,\mathrm{d}x=\l \int_{\mathbb{R}^d} f(x, u) v\,\mathrm{d}x
\end{align*}
for all $v \in \WV$. We define the functionals $\rho, K\colon  \WV \rightarrow \mathbb{R}$ by
\begin{align*}
	\rho(u)= \int_{\RN} \h \L(x, |\nabla u| \r)\,\mathrm{d}x + \int_{\RN} \h_V \L(x, |u| \r)\,\mathrm{d}x
	\quad\text{and}\quad
	K(u)=\int_{\mathbb{R}^N} F(x, u)\,\mathrm{d}x.
\end{align*}

We are now prepared to examine the key properties of the  functionals $\rho$ and $K$.

\begin{theorem}\label{op2}
	Let hypotheses \eqref{H} and \eqref{V0}  be satisfied. Then the functional $\rho$ is well-defined and of class $C^1$ with
	\begin{equation}\label{GFD}
		\begin{aligned}
			\scal{\rho'(u), v}
			&=  \int_{\RN}^{} \L(|\nabla u|^{p\L(x, |\nabla u| \r)-2}\nabla u \cdot\nabla v +\mu (x) \L(|\nabla u|^{q\L(x, |\nabla u| \r)-2}\nabla u\cdot \nabla v \r)\r)\,\mathrm{d}x\\
			&\quad + \int_{\RN}^{}V(x) \L(| u|^{p\L(x, |u| \r)-2} u  v +\mu (x) \L( |u|^{q\L(x, | u| \r)-2} u \nabla v \r)\r)\,\mathrm{d}x.
		\end{aligned}
	\end{equation}
	Moreover, the operator $\rho'$ has the following properties:
	\begin{enumerate}
		\item[\textnormal{(i)}]
			The operator $\rho'\colon \WV \rightarrow \left(\WV\right)^*$ is continuous, bounded, and strictly monotone.
		\item[\textnormal{(ii)}]
			The operator $\rho'$ fulfills the $\left(\mathrm{S}_{+}\right)$-property, i.e.,
			\begin{align*}
				u_n \rightharpoonup u \text { in } \WV \quad \text{and} \quad \limsup _{n \rightarrow \infty}\left\langle \rho'\left(u_n\right), u_n-u\right\rangle \leq 0,
			\end{align*}
			imply $u_n \rightarrow u$ in $\WV$.
		\item[\textnormal{(iii)}]
			The operator $\rho'$  is a homeomorphism.
		\item[\textnormal{(iv)}]
			The operator $\rho'$ is strongly coercive, that is,
			\begin{align*}
				\lim_{\|u\|_{\WV}\rightarrow+\infty} \dfrac{\langle \rho' (u), u\rangle}{\|u\|_{\WV}}\to+\infty.
			\end{align*}
	\end{enumerate}
\end{theorem}

\begin{proof}
	The formula \eqref{GFD} can be derived in a manner similar to the proof of \cite[Proposition 3.16]{Bahrouni-Bahrouni-Missaoui-Radulescu-2024}. The rest of the proof follows with similar arguments as in \cite[Proposition 3.17]{Bahrouni-Bahrouni-Missaoui-Radulescu-2024}.
\end{proof}

\begin{proposition}\label{Ku}
	Let \eqref{H},  \eqref{V0} and \eqref{F} be satisfied. Then, the following hold:
	\begin{enumerate}
		\item[\textnormal{(i)}]
			The functional $K\colon \WV \rightarrow \mathbb{R}$  is of class $C^1$ with
			\begin{align*} \left\langle K^{\prime}(u), v \right\rangle = \int_{\mathbb{R}^d} f(x, u) v \,\mathrm{d}x
			\end{align*}
			for all $u, v \in \WV$.
		\item[\textnormal{(ii)}]
			The functional $J_\l= \rho-\l K$ is of class $C^1$ with
			\begin{align*}
				\left\langle J_\l^{\prime}(u), v \right\rangle=\left\langle \rho^{\prime}(u), v \right\rangle-\l\left\langle K^{\prime}(u), v \right\rangle\quad\text{for all } u,v \in \WV.
			\end{align*}
	\end{enumerate}
\end{proposition}

\begin{remark}
	From Proposition \ref{Ku}, it follows that the solutions of \eqref{prb} correspond to the critical points of the Euler-Lagrange energy functional $J_\lambda$.
\end{remark}

First, we present the following lemma that will be used in the proof of the main existence result.

\begin{lemma}\label{crmi}
	Let the assumptions \eqref{H}, \eqref{V0} and \eqref{F} be satisfied. Then, the functional $J_\lambda=\rho-\l K$ satisfies the \textnormal{(C)}-condition for all $\lambda>0$.
\end{lemma}

\begin{proof}
	Let $\left\{u_n\right\}_{n \in \mathbb{N}} \subseteq \WV$ be a sequence such that $\left(\mathrm{C}_1\right)$ and $\left(\mathrm{C}_2\right)$ from Definition \ref{dcrmi} hold. We divide the proof into two steps.\\
	\textbf{Step 1.} We prove that $\curly{u_n}_{n \in \N}$ is bounded in $\WV$.

	First, from $\left(\mathrm{C}_1\right)$ we have that there exists a constant $M>0$ such that for all $n \in \mathbb{N}$ one has $\left|J_\lambda\left(u_n\right)\right| \leq M$, so
	\begin{align*}
		\left|\int_{\RN}\left(\h\L(x, |\nabla u_n| \r)+V(x)\h\L(x, |u_n|\r)\right)\,\mathrm{d}x-\lambda \int_{\RN} F\left(x, u_n\right)\,\mathrm{d}x\right| \leq M,
	\end{align*}
	which implies that
	\begin{align}\label{cr0}
		\rho(u_n)-\lambda \int_{\RN} F\left(x, u_n\right)\,\mathrm{d}x \leq M\quad\text{for all }n \in \mathbb{N}.
	\end{align}
	Besides, from $\left(\mathrm{C}_2\right)$, there exists $\left\{\varepsilon_n\right\}_{n \in \mathbb{N}}$ with $\varepsilon_n \rightarrow 0^{+}$such that
	\begin{align}\label{eq0001}
		\left|\left\langle J_\lambda^{\prime}\left(u_n\right), v\right\rangle\right| \leq \frac{\varepsilon_n\|v\|_{\WV}}{1+\left\|u_n\right\|_{\WV}} \quad \text {for all } n \in \mathbb{N} \text { and for all } v \in \WV.
	\end{align}
	Choosing $v=u_n$, one has
	\begin{align*}
		\left|\int_{\RN}\left(a(x,|\nabla u_n|)|\nabla u_n|^2+V(x)a(x,|u_n|)|u_n|^2\right)\,\mathrm{d}x-\lambda \int_{\RN} f\left(x, u_n\right) u_n\,\mathrm{d}x\right|\leq \varepsilon_n,
	\end{align*}
	which, multiplied by $\frac{-1}{q+}$, leads to
	\begin{align*}
		-\frac{1}{q+}\int_{\RN}\left(a(x,|\nabla u_n|)|\nabla u_n|^2+V(x)a(x,|u_n|)|u_n|^2\right)\,\mathrm{d}x+\frac{\lambda}{q+} \int_{\RN} f\left(x, u_n\right) u_n\,\mathrm{d}x \leq c_1,
	\end{align*}
	for some $c_1>0$ and for all $n \in \mathbb{N}$. Invoking \eqref{l22}, we conclude that
	\begin{align}\label{cr1}
		-\rho(u_n)+\frac{\lambda}{q+}\int_{\RN} f\left(x, u_n\right)u_n\,\mathrm{d}x \leq c_1\quad\text{for all }n \in \mathbb{N}.
	\end{align}
	Adding \eqref{cr0} and \eqref{cr1} we obtain
	\begin{align}\label{eq0002}
		C& \geq \int_{\RN}\left[ \frac{1}{q^+}f(x,u_n)u_n-F(x,u_n)\right]\,\mathrm{d}x=  \int_{\RN}\tilde{F}(x,u_n)\,\mathrm{d}x
	\end{align}
	for all $n \in \N$ with some constant $C > 0$.

	Arguing by contradiction, we assume that $\Vert u_n\Vert_{\WV}\to+\infty$. Then $\Vert u_n\Vert_{\WV} \geq 1$ for $n$ large enough. Let $v_n=\dfrac{u_n}{\Vert u_n\Vert_{\WV}}\in \WV $, so $\Vert v_n\Vert_{\WV}=1$ and, up to subsequence, we can assume that
	\begin{align*}
		v_n\rightharpoonup v\quad \text{in } \WV
		\quad\text{and}\quad
		v_n(x)\rightarrow v(x)\quad \text{a.e.\,in } \RN.
	\end{align*}
	Note that, exploiting Propositions \ref{nfct} and \ref{RELHV}, we find that, for $n$ large enough,
	\begin{align*}
		\langle J_{\l}^{'}(u_n),u_n\rangle
		& = \scal{\rho(u_n),u_n} -\l\int_{\RN}f(x,u_n)u_n\,\mathrm{d}x\\
		& = \int_{\RN}a(x,\vert \nabla u_n\vert)|\nabla u_n|^2\,\mathrm{d}x +\int_{\RN}V(x)a(x, |u_n|) | u_n|^2\,\mathrm{d}x   -\l\int_{\RN}f(x,u_n)u_n\,\mathrm{d}x\\
		& \geq p^-\rho(u_n) -\l\int_{\RN}f(x,u_n)u_n\,\mathrm{d}x\\
		& \geq \Vert u_n\Vert_{\WV}^{p^-}-\l\int_{\RN}f(x,u_n)u_n\,\mathrm{d}x,
	\end{align*}
	since $\Vert u_n\Vert_{\WV} \geq 1$. Thus
	\begin{align}\label{eq0004}
		\frac{\langle J_{\l}^{'}(u_n),u_n\rangle}{\Vert u_n\Vert_{\WV}^{p^-}}\geq 1-\int_{\RN}\frac{f(x,u_n)}{\Vert u_n\Vert_{\WV}^{p^-}}u_n\,\mathrm{d}x.
	\end{align}
	From $(\ref{eq0001})$ and $(\ref{eq0004})$ it follows that
	\begin{align}\label{eq0005}
		\limsup\limits_{n\rightarrow+\infty}\int_{\RN}\frac{f(x,u_n)}{\Vert u_n\Vert_{\WV}^{p^-}}u_n\,\mathrm{d}x\geq 1.
	\end{align}
	For $r\geq 0$ we set
	\begin{align*}
		\mathfrak{F}(r):=\inf\left\lbrace \tilde{F}(x,s)\colon  x\in\RN \text{ and } s\in \mathbb{R} \text{ with }  s \geq r\right\rbrace.
	\end{align*}
	By \eqref{F}(ii)-(iv), we have
	\begin{align}\label{fr0}
		\mathfrak{F}(r)>0\quad \text{for all } r \text{ large}
		\quad\text{and}\quad
		\mathfrak{F}(r)\rightarrow +\infty\quad \text{as }  r\rightarrow+\infty.
	\end{align}
	For $0\leq a<b\leq +\infty$ let
	\begin{align*}
		A_n(a,b)&:=\left\lbrace x\in \RN\colon  a\leq \vert u_n(x)\vert <b\right\rbrace,\\
		c_a^b&:=\inf\left\lbrace \frac{\tilde{F}(x,s)}{\vert s\vert^{p^-}}\colon  x\in\RN \text{ and } s\in \mathbb{R}\setminus\{0\} \text{ with } a\leq \vert s\vert<b\right\rbrace.
	\end{align*}
	Note that
	\begin{align*}
		\tilde{F}(x,u_n)\geq c_a^b\vert u_n\vert^{p^-}\quad \text{for all } x\in A_n(a,b).
	\end{align*}
	It follows from $(\ref{eq0002})$ that
	\begin{equation}\label{eq0007}
		\begin{aligned}
			C & \geq \int_{\RN}\tilde{F}(x,u_n)\,\mathrm{d}x\\
			& =\int_{A_n(0,a)}\tilde{F}(x,u_n)\,\mathrm{d}x+\int_{A_n(a,b)}\tilde{F}(x,u_n)\,\mathrm{d}x+\int_{A_n(b,+\infty)}\tilde{F}(x,u_n)\,\mathrm{d}x\\
			& \geq\int_{A_n(0,a)}\tilde{F}(x,u_n)\,\mathrm{d}x+c_a^b\int_{A_n(a,b)}\vert u_n\vert^{p^-}\,\mathrm{d}x+\mathfrak{F}(b)\vert A_n(b,+\infty)\vert
		\end{aligned}
	\end{equation}
	for $b$ large enough. Using Theorem \ref{Inj30}, we get $\gamma_r>0$ such that $\Vert v_n\Vert^r_{L^r(\RN)}\leq \gamma_r\Vert v_n\Vert^r_{\WV}=\gamma_3$ with $p^-\leq r<p^-_*$. Let $0<\varepsilon<\frac{1}{3}$. By assumption \eqref{F}(iii), there exists $a_\varepsilon>0$ such that
	\begin{align}\label{eq0008}
		\vert f(x,s)\vert \leq \frac{\varepsilon}{3\gamma_{p^-}}\vert s\vert^{p^--1}\quad\text{for all } \vert s\vert \leq a_\varepsilon.
	\end{align}
	From $(\ref{eq0008})$ and Theorem \ref{Inj30}, we obtain
	\begin{equation}\label{eq0009}
		\begin{aligned}
			\int_{A_n(0,a_\varepsilon)}\frac{f(x,u_n)}{\Vert u_n\Vert^{p^-}}u_n\,\mathrm{d}x
			& \leq \frac{\varepsilon}{3\gamma_{p^-}} \int_{A_n(0,a_\varepsilon)}\frac{\vert u_n\vert^{p^-}}{\Vert u_n\Vert^{p^-}}\,\mathrm{d}x\\
			& \leq \frac{\varepsilon}{3\gamma_{p^-}} \int_{A_n(0,a_\varepsilon)}\vert v_n\vert^{p^-}\,\mathrm{d}x\\
			& \leq \frac{\varepsilon}{3\gamma_{p^-}} \gamma_{p^-} \Vert v_n\Vert^{p^-}\\
			& = \frac{\varepsilon}{3}\quad \text{for all } n\in \mathbb{N}.
		\end{aligned}
	\end{equation}
	Now, exploiting $(\ref{eq0007})$ and assumption \eqref{F}(iv), we see that
	\begin{align*}
		C\geq \int_{A_n(b,+\infty)}\tilde{F}(x,u_n)\,\mathrm{d}x\geq \mathfrak{F}(b)\vert A_n(b,+\infty)\vert.
	\end{align*}
	It follows, using \eqref{fr0}, that
	\begin{align}\label{eq00010}
		\vert A_n(b,+\infty)\vert\rightarrow 0\quad \text{as } b\rightarrow+\infty \text{ uniformly in } n.
	\end{align}
	Set $\displaystyle{\sigma'=\frac{\sigma}{\sigma-1}}$ where $\sigma$ is defined in \eqref{F}(iv). Since $\sigma>\frac{d}{p^-}$, one sees that $p^-\sigma^{'}\in(p^-,p^-_*)$.

	Let $\tau\in (p^-\sigma^{'},p^-_*)$. Using Theorem \ref{Inj30}, H\"older's inequality and $(\ref{eq00010})$, for $b$ large, we find
	\begin{equation}\label{eq00020}
		\begin{aligned}
			\left( \int_{A_n(b,+\infty)}\vert v_n\vert ^{p^-\sigma^{'}}\,\mathrm{d}x \right)^{\frac{1}{\sigma^{'}}}
			& \leq \vert A_n(b,+\infty)\vert ^{\frac{\tau-p^-\sigma^{'}} {\tau\sigma^{'}}}\left( \int_{A_n(b,+\infty)}\vert v_n\vert ^{p^-\sigma^{'}\frac{\tau}{p^-\sigma^{'}}}\,\mathrm{d}x\right)^{\frac{p^-}{\tau}} \\
			& \leq \vert A_n(b,+\infty)\vert ^{\frac{\tau-p^-\sigma^{'}}{\tau\sigma^{'}}}\left( \int_{A_n(b,+\infty)}\vert v_n\vert ^{\tau}\,\mathrm{d}x\right)^{\frac{p^-}{\tau}} \\
			& \leq \vert A_n(b,+\infty)\vert ^{\frac{\tau-p^-\sigma^{'}}{\tau\sigma^{'}}}\gamma_{\tau}\Vert v_n\Vert^{p^-}\\
			& = \vert A_n(b,+\infty)\vert ^{\frac{\tau-p^-\sigma^{'}}{\tau\sigma^{'}}}\gamma_{\tau}\\
			& \leq \frac{\varepsilon}{3\L(\tilde{c}C\r)^{\frac{1}{\sigma}}}\quad \text{uniformly in } n.
		\end{aligned}
	\end{equation}
	By \eqref{F}(iv), H\"older's inequality, $(\ref{eq0007})$ and $(\ref{eq00020})$, we can choose $b_\varepsilon\geq r_0$ large so that
	\begin{equation}\label{eq00011}
		\begin{aligned}
			\int_{A_n(b_\varepsilon,+\infty)}\L|\frac{f(x,u_n)}{\Vert u_n\Vert^{p^-}}u_n\r|\,\mathrm{d}x
			& \leq\int_{A_n(b_\varepsilon,+\infty)}\frac{|f(x,u_n)|}{\vert u_n\vert^{p^--1}}\vert v_n\vert ^{p^-}\,\mathrm{d}x\\
			& \leq \left( \int_{A_n(b_\varepsilon,+\infty)}\left\lvert\frac{f(x,u_n)}{\vert u_n\vert^{p^--1}}\right\rvert^{\sigma}\,\mathrm{d}x\right)^{\frac{1}{\sigma}} \left( \int_{A_n(b_\varepsilon,+\infty)}\vert v_n\vert ^{p^-\sigma^{'}}\,\mathrm{d}x\right)^{\frac{1}{\sigma^{'}}} \\
			& \leq \left( \tilde{c}\int_{A_n(b_\varepsilon,+\infty)}\tilde{F}(x,u_n)\,\mathrm{d}x\right)^{\frac{1}{\sigma}} \left( \int_{A_n(b_\varepsilon,+\infty)}\vert v_n\vert ^{p^-\sigma^{'}}\,\mathrm{d}x\right)^{\frac{1}{\sigma^{'}}} \\
			& \leq \frac{\varepsilon}{3}\quad \text{uniformly in } n.
		\end{aligned}
	\end{equation}
	Next, from $(\ref{eq0007})$, we have
	\begin{equation}\label{eq00014}
		\begin{aligned}
			\int_{A_n(a,b)}\vert v_n\vert^{p^-}\,\mathrm{d}x
			& =\frac{1}{\Vert u_n\Vert_{\WV}^{p^-}}\int_{A_n(a,b)}\vert u_n\vert^{p^-}\,\mathrm{d}x\\
			&\leq \frac{C}{c_a^b \Vert u_n\Vert_{\WV}^{p^-}}\to 0\quad \text{as }  n\to +\infty.
		\end{aligned}
	\end{equation}
	Since $\displaystyle{\frac{f(x,s)}{ \vert s\vert^{p^--1}}}$ is continuous on $a\leq \vert s\vert\leq b$,  there exists $c>0$ depending on $a$ and $b$ and independent from $n$, such that
	\begin{align}\label{eq00012}
		\vert f(x,u_n)\vert \leq c\vert u_n\vert^{p^--1}\quad \text{for all } x\in A_n(a,b).
	\end{align}
	Using $(\ref{eq00014})$ and $(\ref{eq00012})$, we can choose $n_0$ large enough such that
	\begin{equation}\label{eq00013}
		\begin{aligned}
			\int_{A_n(a_\varepsilon,b_\varepsilon)}\L|\frac{f(x,u_n)}{\Vert u_n\Vert_{\WV}^{p^-}}u_n \r|\,\mathrm{d}x
			& \leq\int_{A_n(a_\varepsilon,b_\varepsilon)}\frac{|f(x,u_n)|}{\vert u_n\vert^{p^--1}}\vert v_n\vert^{p^-}\,\mathrm{d}x\\
			& \leq c\int_{A_n(a_\varepsilon,b_\varepsilon)}\vert v_n\vert^{p^-}\,\mathrm{d}x\\
			& \leq c\frac{C}{c_{a_\varepsilon}^{b_\varepsilon} \Vert u_n\Vert_{\WV}^{p^-}}\\
			& \leq \frac{\varepsilon}{3}\quad \text{for all } n\geq n_0.
		\end{aligned}
	\end{equation}
	Combining \eqref{eq0009}, \eqref{eq00011} and \eqref{eq00013}, we find that
	\begin{align*}
		\int_{\RN}\frac{f(x,u_n)}{\Vert u_n\Vert_{\WV}^{p^-}}u_n\,\mathrm{d}x\leq \varepsilon\quad \text{for all } n\geq n_0,
	\end{align*}
	which contradicts to \eqref{eq0005}. Therefore, $\lbrace u_n\rbrace_{n\in\mathbb{N}}$ is bounded in $\WV$.\\
	\textbf{Step 2.} $u_n \to u $ in $\WV$ as $n\rightarrow +\infty$ up to a subsequence.

	Since $\left\{u_n\right\}_{n \in \mathbb{N}} \subset \WV$ is bounded by Step 1 and $\WV$ is a reflexive space, there exists a subsequence, not relabeled, that converges weakly in $\WV$ and strongly in $L^{b^+}(\RN)$ (see Theorem \ref{Inj30}), that is,
	\begin{align*}
		u_n \rightharpoonup u \quad \text {in } \WV
		\quad\text{and}\quad
		u_n \rightarrow u \quad \text {in } L^{b^+}(\RN).
	\end{align*}
	Note that, from \eqref{F}(i), we have $f(x,t)\leq c|t|^{b^+-1}$.
	Therefore, using this in \eqref{eq0001} with $v=u_n-u$ and passing to the limit as $n \rightarrow +\infty$, we obtain
	\begin{align*}
		\left\langle \rho'\left(u_n\right), u_n-u\right\rangle \rightarrow 0 \quad \text {as } n \rightarrow +\infty.
	\end{align*}
	Since $\rho'$ satisfies the $\left(S_{+}\right)$-property, see Theorem \ref{op2}(ii), we get the desired assertion of the theorem.
\end{proof}

\begin{proof}[Proof of Theorem \ref{thex}]
	First, we assume that the assumptions of Theorem \ref{thex} are satisfied. Before proving the main existence result, we introduce some notation. We fix an open ball centered at $ x_0 \in \RN$ with radius $ R $, which we denote by $ B(x_0, R) $, and we set $\omega_R$ as the Lebesgue measure of the ball $B(x_0, R)$ in $\mathbb{R}^d$, which is given by
	\begin{align*}
		\omega_R:= \L | B(x_0, R) \r|= \frac{\pi^{\frac{d}{2}}}{\Gamma(1+\frac{d}{2})} R^d.
	\end{align*}
	Next, let
	\begin{align}\label{thta}
		\delta:=\dfrac{p^-\min \left\{R^{p^{-}}, R^{q^{+}}\right\}}{\max \left\{1,\|\mu\|_{\infty}\right\} \omega_R \L ( V_\infty \min \left\{R^{p^{-}}, R^{q^{+}}\right\} + 2^{q^{+}+1-d}(2^d-1)\r)},
	\end{align}
	where $V_\infty$ will be defined later. Further, we define
	\begin{align}
		\alpha(r)&:=\frac{\bar{\gamma}_{\mathcal{B}} \max \left\{\left(q^{+} r\right)^{\frac{b^{+}}{p^{-}}},\left(q^{+} r\right)^{\frac{b^-}{q_{+}}}\right\} }{r},\label{ar}\\
		\beta(\eta)&:=\delta \frac{\int_{B\left(x_0, \frac{R}{2}\right)} F(x, \eta)\,\mathrm{d}x}{\max \left\{\eta^{p_{-}}, \eta^{q_{+}}\right\}},\label{be}
	\end{align}
	where $\bar{\gamma}_{\mathcal{B}}=\max \left\{\gamma_{\mathcal{B}}^{b^+}, \gamma_{\mathcal{B}}^{b^-}\right\}$ and $ b $ is given in  \eqref{F}(i).

	Now, let $\rho$ and $K$ be as given in Subsection \ref{Prty}. First we see that $\rho$ and $K$ fulfill all the required regularity properties in Theorem \ref{bonano}. Indeed, $\rho$ is coercive due to Proposition \ref{RELHV}(iv) and the functional $J_\lambda$ is unbounded from below because of \eqref{F}(ii). Now, fix $\l \in \Lambda$, which is nonempty because of $(\mathrm{H}_2)$, and consider $\tilde{u} \in \WV$ defined by
	\begin{align*}
		\tilde{ u}(x):=
		\begin{cases}
			0&  \text{if } x \in \RN \backslash B\L(x_0,R\r), \\
			u_R:=  \dfrac{2 \eta}{R}(R- |x-x_0|) & \text{if } x \in B\L(x_0,R\r)\backslash B\L(x_0,\frac{R}{2}\r)=:S,  \\
			\eta &  \text{ if }x \in  B\L(x_0,\frac{R}{2}\r).
		\end{cases}
	\end{align*}
	First, for $x\in S$, we have that
	\begin{equation}\label{ur1}
		\begin{aligned}
			|u_R|=\L|\dfrac{2 \eta}{R}(R- |x-x_0|)\r|
			& =\L|\dfrac{2 \eta}{R}\r| \L|(R- |x-x_0|)\r| \\
			& \leq \L|\dfrac{2 \eta}{R}\r| \L|\frac{R}{2}\r|=\eta.
		\end{aligned}
	\end{equation}
	Moreover, since $V\in C(\RN, \R)$, we have $V \in L^\infty(B\L(x_0,R\r))$ and so we set $V_\infty:= \|V\|_{L^\infty(B\L(x_0,R\r))}$.\\
	\textbf{Step 1.} $0<\rho(\tilde{u})<r$.

	To simplify, we rewrite the functional $\rho$ as
	\begin{align*}
		\rho(\tilde{u})=\rho_{\h}( \nabla \tilde{u})+\rho_{\h_V}(\tilde{u}),
	\end{align*}
	where
	\begin{align*}
		\rho_{\h}( \nabla \tilde{u})= \int_{\RN}^{} \h(x, |\nabla \tilde{u}|)\,\mathrm{d}x
		\quad\text{and}\quad
		\rho_{\h_V}(\tilde{u})= \int_{\RN}^{} \h_V(x, |\tilde{u}|)\,\mathrm{d}x.
	\end{align*}
	Using \eqref{l22} and proceeding as in \cite[Theorem 3.2, p. 743]{Amoroso-Bonanno-DAgui-Winkert-2024}, we easily show that
	\begin{equation}\label{i1}
		\begin{aligned}
			\rho_{\h}( \nabla \tilde{u})
			& =\int_{\RN}^{} \h(x, |\nabla \tilde{u}|)\,\mathrm{d}x \\&\leq \frac{1}{p^-} \int_{\RN}^{} h(x, |\nabla \tilde{u}|) |\nabla \tilde{u}|\,\mathrm{d}x \\
			&=\frac{1}{p^-} \int_{S}\left(\left(\frac{2 \eta}{R}\right)^{p(x,\frac{2 \eta}{R})}+\mu(x)\left(\frac{2 \eta}{R}\right)^{q(x,\frac{2 \eta}{R})}\right)\,\mathrm{d}x \\
			&\leq \frac{2^{q^{+}+1-d}(2^d-1)}{p^{-}}  \frac{\max \left\{1,\|\mu\|_{\infty}\right\}}{\min \left\{R^{p^{-}}, R^{q^{+}}\right\}} \max \left\{\eta^{p ^{-}}, \eta^{q^{+}}\right\} \omega_R.
		\end{aligned}
	\end{equation}
	Again, by \eqref{l22} and invoking \eqref{ur1}, it yields that
	\begin{equation}\label{i2}
		\begin{aligned}
			\rho_{\h_V}(\tilde{u})
			& =\int_{\RN}^{} \h_V(x, | \tilde{u}|)\,\mathrm{d}x \leq \frac{1}{p^-} \int_{\RN}^{}V(x) h(x, | \tilde{u}|) | \tilde{u}|\,\mathrm{d}x \\
			&=\frac{1}{p^-} \int_{S} V(x)\L( |u_R|^{p(x,|u_R|)} +\mu(x) |u_R|^{q(x,|u_R|)}   \r)\,\mathrm{d}x\\
			&\quad +\frac{1}{p^-} \int_{ B\left(x_0, \frac{R}{2}\right)}V(x)\left(\eta^{p(x,\eta )}+\mu(x) \eta^{q(x,\eta)}\right)\,\mathrm{d}x \\
			&\leq \frac{1}{p^-} \int_{S}V(x)\left( \eta^{p(x,\eta )}+\mu(x) \eta^{q(x,\eta)} \right)\,\mathrm{d}x\\
			&\quad +\frac{1}{p^-} \int_{ B\left(x_0, \frac{R}{2}\right)}V(x)\left( \eta^{p(x,\eta )}+\mu(x)\eta^{q(x,\eta)}\right)\,\mathrm{d}x\\
			& = \frac{1}{p^-}  \int_{ B\left(x_0, R\right)}V(x)\left( \eta^{p(x,\eta )}+\mu(x)\eta^{q(x,\eta)}\right)\,\mathrm{d}x \\
			&\leq \frac{V_\infty}{p^-}   \max \left\{1,\|\mu\|_{\infty}\right\}\max \left\{\eta^{p^{-}}, \eta^{q^{+}}\right\} \omega_R.
		\end{aligned}
	\end{equation}
	Therefore, adding \eqref{i1} and \eqref{i2} and taking \eqref{318} in mind, it follows  that
	\begin{equation}\label{rho}
		\begin{aligned}
			\rho(\tilde{u})
			&\leq \max \left\{\eta^{p^{-}}, \eta^{q^{+}}\right\}\L(\frac{1}{p^-} \max \left\{1,\|\mu\|_{\infty}\right\} \omega_R \L ( V_\infty + \frac{2^{q^{+}+1-d}(2^d-1)}{\min \left\{R^{p_{-}}, R^{q_{+}}\right\}}  \r) \r)\\
			&=\frac{1}{\delta}\max \left\{\eta^{p^{-}}, \eta^{q^{+}}\right\}<r.
		\end{aligned}
	\end{equation}
	\noindent\textbf{Step 2.} We need to verify the validity of condition \eqref{32}.

	From assumption $\left(\mathrm{H}_1\right)$ and \eqref{ur1}, we obtain
	\begin{align*}
		K(\tilde{u})=\int_{B\left(x_0, \frac{R}{2}\right)} F(x, \eta)\,\mathrm{d}x+\int_{S} F\left(x, \frac{2 \eta}{R}\left(R-\left|x-x_0\right|\right)\right)\,\mathrm{d}x \geq \int_{B\left(x_0, \frac{R}{2}\right)} F(x, \eta)\,\mathrm{d}x.
	\end{align*}
	Hence, in view of \eqref{rho}, we infer that
	\begin{align}\label{ml11}
		\frac{K(\tilde{u})}{\rho(\tilde{u})} \geq \delta \frac{\ds\int_{B\left(x_0, \frac{R}{2}\right)} F(x, \eta)\,\mathrm{d}x}{\max \left\{\eta^{p_{-}}, \eta^{q_{+}}\right\}}.
	\end{align}
	On the other hand, fix $u \in \WV$ such that $\rho(u) \leq r $. Then, invoking Proposition \ref{RELHV}, one has
	\begin{align}\label{rho1}
		q^+r\geq q^+\rho(u)> \rho (u) \geq \min \L\{ \|u\|_{\WV}^{p^-}, \|u\|_{\WV}^{q^+} \r\}.
	\end{align}
	Thus, by \eqref{F}(i), Proposition \ref{zoo} and \eqref{rho1}, we get
	\begin{align*}
		\sup _{\left.u \in \rho^{-1}(]-\infty, r]\right)} K(u)
		&\leq \sup _{\left.u \in \rho^{-1}(]-\infty, r]\right)}  \int_{\RN}\mathcal{B}\left(x,|u|\right)\,\mathrm{d}x \\
		& =\sup _{\left.u \in \rho^{-1}(]-\infty, r]\right)} \rho_{\mathcal{B}}(u) \leq \sup _{\left.u \in \rho^{-1}(]-\infty, r]\right)} \left(\max \left\{\|u\|_{L^{\mathcal{B}}(\RN)}^{b^{-}},\|u\|_{L^{\mathcal{B}}(\RN)}^{b^{+}}\right\}\right) \\
		& \leq \sup _{\left.u \in \rho^{-1}(]-\infty, r]\right)}  \left(\max \left\{\gamma_{\mathcal{B}}^{b^+}, \gamma_{\mathcal{B}}^{b^-}\right\} \max \left\{\|u\|_{\WV}^{b^{-}},\|u\|_{\WV}^{b^{+}}\right\}\right) \\
		& \leq \left(\bar{\gamma}_{\mathcal{B}} \max \left\{\left(q^{+}  r\right)^{\frac{b^{+}}{p^{-}}},\left(q^{+} r\right)^{\frac{b^-}{q_{+}}}\right\}\right).
	\end{align*}
	Then, taking $\left(\mathrm{H}_2\right)$ and \eqref{ml11} into account, we get
	\begin{align*}
		\frac{\ds \sup _{\left.u \in \rho^{-1}(]-\infty, r]\right)} K(u)}{r}
		&\leq \frac{\left(\bar{\gamma}_{\mathcal{B}} \max \left\{\left(q^{+} r\right)^{\frac{b^{+}}{p^{-}}},\left(q^{+} r\right)^{\frac{b^-}{q^{+}}}\right\}\right) }{r}=\alpha (r)< \beta (\eta)\\
		&= \delta \frac{\ds\int_{B\left(x_0, \frac{R}{2}\right)} F(x, \eta)\,\mathrm{d}x}{\max \left\{\eta^{p_{-}}, \eta^{q_{+}}\right\}} \leq \frac{K(\tilde{u})}{\rho(\tilde{u})}.
	\end{align*}
	This proves Step 2.

	From Steps 1 and 2 and Lemma \ref{crmi} we see that all conditions in Theorem \ref{bonano} are satisfied and so we conclude that problem \eqref{prb} has at least two nontrivial weak solutions $u_{\lambda, 1}, u_{\lambda, 2} \in \WV$ such that $J_\lambda\left(u_{\lambda, 1}\right)<0<J_\lambda\left(u_{\lambda, 2}\right)$. This finishes the proof.
\end{proof}

\section{Concluding remarks, perspectives, and open problems}

\begin{enumerate}
	\item[\textnormal{(i)}]
		Building on the key results established in this paper (Theorems \ref{Inj1}, \ref{Inj2}, \ref{Inj3}, and \ref{Inj30}) together with those obtained by Bahrouni--Bahrouni--Missaoui--R\u{a}dulescu \cite{Bahrouni-Bahrouni-Missaoui-Radulescu-2024}, we believe that a comprehensive investigation of equation \eqref{prb} is within reach, including aspects such as existence, uniqueness, multiplicity, and regularity.
	\item[\textnormal{(ii)}]
		In a forthcoming paper, we will investigate the multiplicity and uniqueness of solutions for the non-variational case of problem \eqref{prb}. In particular, we will study problem \eqref{prb} involving the gradient term $\nabla u$ (or convection term) in the nonlinear data.
	\item[\textnormal{(iii)}]
		Throughout this paper, we focused on cases where the exponent depends on the gradient of the solution. A natural extension of this work would be to explore what happens if, instead of the gradient, the exponent depends directly on the solution itself. Specifically, we consider the following equation:
		\begin{equation} \label{u}
			\begin{aligned}
				-\operatorname{div}&\left(|\nabla u|^{p(x,| u|)-2} \nabla u+ \mu(x)| \nabla u|^{q(x,| u|)-2}  \nabla u\right)\\&+ V(x)\left(|u|^{p(x,| u|)-2}  u+ \mu(x)|  u|^{q(x,| u|)-2}  u\right)=\l f(x,u), \ \ x \in \mathbb{R}^d.
			\end{aligned}
		\end{equation}
		As described in detail in \cite{Bahrouni-Bahrouni-Missaoui-Radulescu-2024}, the key difference between the equation above and equation \eqref{prb} lies in the absence of appropriate Musielak-Orlicz Sobolev spaces. As a result, the proof of Theorem \ref{thex} cannot be directly extended to equation \eqref{u}. The original approach for addressing problem \eqref{u} involves constructing an auxiliary problem and then demonstrating the existence of a solution to equation \eqref{u} by passing to the limit. We believe that for the auxiliary problem introduced in \cite{Bahrouni-Bahrouni-Missaoui-Radulescu-2024}, it is possible to establish the existence of two solutions using methods analogous to those in the proof of Theorem \ref{thex}. However, when passing to the limit, these two solutions may coincide. It is therefore interesting to investigate whether Theorem \ref{thex} still applies to equation \eqref{u}.
	\item[\textnormal{(iv)}]
		Note that the proofs of Theorems \ref{Inj2}, \ref{Inj3}, and \ref{Inj30} rely heavily on Theorem \ref{Inj1}. However, the proof of Theorem \ref{Inj1} is specific to the particular N-function $\mathcal{H}$. Therefore, it is important to check whether the same results can be obtained for a more general N-function.
	\item[\textnormal{(v)}]
		We would like to highlight that the result obtained in Theorem \ref{Inj30} is not optimal. Specifically, we established the compact embedding of $W^{1,\mathcal{H}}_V(\mathbb{R}^d)$ into a Lebesgue space with a variable exponent. To achieve the optimal result, and following the approach used by Berger--Schechter \cite{Berger-Schechter-1972}, a new compact embedding theorem for $W^{1,\mathcal{H}}(\Omega)$ is required, where $\Omega$ is a specific unbounded domain. Proving this result involves considerable effort, so we consider it an open problem.
\end{enumerate}


\end{document}